\documentclass[11pt,reqno]{amsart}
\usepackage{amsmath, amsthm, amsfonts}
\usepackage{amssymb}
\usepackage{amscd}
\usepackage{verbatim}
\numberwithin{equation}{section}

\newtheorem{theorem}{Theorem}[section]
\newtheorem{corollary}[theorem]{Corollary}

\theoremstyle{definition}
\newtheorem{definition}[theorem]{Definition}
\newtheorem{remark}[theorem]{Remark}
\newtheorem{example}[theorem]{Example}
\newcommand{\loc}{{\mathrm{loc}}}
\newcommand{\dx}{\,\mathrm{d}x}
\newcommand{\dy}{\,\mathrm{d}y}

\newcommand{\dt}{\,\mathrm{d}t}

\newcommand{\dnu}{\,\mathrm{d}\nu}

\newcommand{\drho}{\,\mathrm{d}\rho}

\newcommand{\core}{C_0^{\infty}(\Omega)}

\newcommand{\be}{\begin{equation}}
\newcommand{\ee}{\end{equation}}
\newcommand{\bea}{\begin{eqnarray}}
\newcommand{\eea}{\end{eqnarray}}
\newcommand{\bean}{\begin{eqnarray*}}
\newcommand{\eean}{\end{eqnarray*}}

\newcommand{\Real}{\mathbb{R}}

\newcommand{\Nat}{\mathbb{N}}

\newcommand{\opname}[1]{\mbox{\rm #1}\,}
\newcommand{\supp}{\opname{supp}}

\newlength{\wex}  \newlength{\hex}
     \def\gb{\beta}       
             
                         \def\vge{\varepsilon}
\def\gf{\phi}       \def\vgf{\varphi}    
      \def\gk{\kappa}      \def\gl{\lambda}
\def\gm{\mu}        \def\gn{\nu}         
    \def\gr{\rho}        
\def\gs{\sigma}

     \def\Gd{\Delta}      

\def\Gl{\Lambda}          
\def\Gw{\Omega}              
\newcommand{\Green}[4]{\mbox{$G^{#1}_{#2}(#3,#4)$}}

\newcommand{\beq}{\begin{equation}}
\newcommand{\eeq}{\end{equation}}
\newcommand{\beqa}{\begin{eqnarray}}
\newcommand{\eeqa}{\end{eqnarray}}
\newcommand{\beqanl}{\begin{eqnarray*}}
\newcommand{\eeqanl}{\end{eqnarray*}}
\def\squarebox#1{\hbox to #1{\hfill\vbox to #1{\vfill}}}
\newcommand{\Rn}{\Real ^n}

\newcommand{\R}{{\mathbb R}}
\newcommand{\N}{{\mathbb N}}

\catcode`@=11 \@addtoreset{equation}{section} \catcode`@=12
\title[Boundedness and compactness of Green operators]{On the boundedness and compactness of weighted Green operators of second-order elliptic operators}
\author[Yehuda Pinchover]{Yehuda Pinchover}
\address{Yehuda Pinchover,
Department of Mathematics, Technion - Israel Institute of
Technology,   Haifa 32000, Israel}
\email{pincho@techunix.technion.ac.il}

\thanks{Dedicated to Professor Pavel Exner
on the occasion of his 70th birthday. 
The author is grateful to
Professor M.~Cwikel and the late professor V.~Liskevich for valuable
discussions. This research was supported by the Israel Science
Foundation (grants No. 963/11) founded by the Israel Academy of
Sciences and Humanities.}

\begin{document}
%\date{}
%\maketitle
%\thispagestyle{empty}
\begin{abstract} For a given second-order linear elliptic operator $L$ which admits a positive minimal Green function, and a given positive weight function $W$, we introduce  a family of weighted Lebesgue spaces $L^p(\phi_p)$ with their dual spaces, where $1\leq p\leq \infty$. We study some fundamental properties of the corresponding (weighted) Green operators on these spaces. In particular, we prove that these Green operators are bounded on $L^p(\phi_p)$ for any $1\leq p\leq \infty$ with a uniform bound. We study the existence of a
principal eigenfunction for these operators in these spaces, and the simplicity of the corresponding principal eigenvalue. We also show that such a Green operator is
a resolvent of a densely defined closed operator which is equal to $(-W^{-1})L$ on $C_0^\infty$, and that this closed operator generates a strongly
continuous contraction semigroup. Finally, we prove that if $W$ is a (semi)small perturbation of $L$, then for any $1\leq p\leq \infty$, the associated Green operator is compact on $L^p(\phi_p)$, and the corresponding spectrum is $p$-independent.\\[2mm]

\noindent  2000  \! {\em Mathematics  Subject  Classification.}
Primary  \! 35J08; Secondary  35B09, 35J15, 35P05, 47D06.\\[2mm]

\end{abstract}

\maketitle

%%%%%%%%%%%%%%%%%%%%%%%%%%%%%%%%%%%%%%%

\section{Introduction}
%%%%%%%%%%%%%%%%%%%%%%%%%%%%%%%%%%%%%%%%%%%%%%%
Positive (Dirichlet) Green functions of second-order linear elliptic operators with real coefficients and their induced integral operators are among the most important building blocks of the elliptic theory for such operators, and in particular, for the qualitative theory of positive solutions of the corresponding homogeneous equations. In many problems, and  in particular in the study of positive solutions, the underling topology is the open compact topology, i.e. the topology of locally compact convergence (e.g., in the Martin boundary theory, in the study of the heat kernel,  and in criticality theory). On the other hand, when dealing with spectral theory of such operators, in the study of semigroups generated by such operators, or in the study of well-posedness of boundary value problems, one should usually specify a relevant Banach space. Frequently, one takes one of the classical Lebesgue spaces as the underlying space, but a priori, it is not clear why these spaces are the appropriate functional spaces to study various elliptic problems for a specific operator in a given domain.

\medspace

In the present paper, we introduce for a given second-order linear elliptic operator $L$ which is defined on a noncompact manifold $\Gw$ and admits a positive Green function, and for a given positive weight function $W$, a family of weighted Lebesgue spaces $L^p(\gf_p)$, where $1\leq p\leq \infty$. The weight $\gf_p$ is given by
$$\gf_p:=\gf^{-1}(\gf W\tilde{\gf})^{1/p},
$$ where $\gf$,  (resp.    $\tilde{\gf}$) is a fixed positive solution of the equation $(L-\gm W)u=0$ (resp.    $(L^\star -\gm W)u=0$) in $\Gw$, and $L^\star$ is the formal adjoint of $L$ (see Section~\ref{sec1} for a detailed discussion on these spaces, and also for the needed terminology and some preliminary results).

Clearly, if the positive Liouville theorem holds for $L-\gm W$ and $L^\star-\gm W$ in $\Gw$, then $L^p(\gf_p)$ is uniquely defined. In addition, $L^1(\gf_1)$ is always independent of $\gf$ while $L^\infty(\gf_\infty)$ is independent of $\tilde{\gf}$ and $W$. In \cite{P99}, $L^\infty(\gf_\infty)$ was introduced, and some properties of weighted Green operators on $L^\infty(\gf_\infty)$ were studied.    Moreover, if $L$ is symmetric, then one may choose $\tilde{\gf}=\gf$, and then
$L^2(\gf_2)=L^2(\Gw,W)$, and hence, this space is $\gf$~independent. We also note that $\{L^p(\gf_p)\}_{p\geq 1}$ is a family of real interpolation spaces. The latter observation is used to prove several results of the paper.

\medskip

The aim of the present paper is to study some fundamental properties of the induced weighted Green operators on these weighted Lebesgue spaces. In particular, we prove in Section~\ref{secgreen} that the corresponding weighted Green operator is bounded on $L^p(\gf_p)$ for any $1\leq p\leq \infty$ with a bound independent of $p$, $\gf$, and $\tilde{\gf}$. In Section~\ref{sec_principal}, we study the existence and uniqueness  of a
principal eigenfunction for these weighted Green operators on $L^p(\gf_p)$, and the simplicity of the corresponding principal eigenvalue for $1\leq p\leq \infty$.

Next, we show in Section~\ref{secsemigr} that for $1\leq p< \infty$, the weighted Green operator is
a resolvent of a densely defined closed operator $A_p$ such that $A_p=-W^{-1}L$ on $\core$. It turns out that under some further assumptions, $A_p$ generates a strongly
continuous contraction semigroup. Finally, in Section~\ref{secanti} we prove that if $W$ is  a {\em (semi)small perturbation} of $L$ in $\Gw$, then for any $1\leq p\leq \infty$, the associated weighted Green operator is compact on $L^p(\gf_p)$, and the corresponding spectrum is $p$-independent. We note that if in addition $L$ is symmetric, then it follows from  \cite{P99} that for any $k\geq 1$ the quotient $\phi_k/\phi$ has a continuous extension up to
the Martin boundary of the pair $(\Omega, L)$, where $\phi$ is
the ground state of $L$ with a principal eigenvalue $\lambda_0$, and $\phi_k$ is a weighted eigenfunction in $L^2(\Gw, W\!\dnu)$. It follows from the $p$-independence of the spectrum that in fact, $\phi, \phi_k \in L^p(\gf_p)$, for all  $1\leq p\leq \infty$.

\medskip

The problem of the $p$-independence of the spectrum of generators of semigroups on $L^p$ was raised by B.~Simon in \cite{Si80} for Schr\"odinger semigroups in $\R^d$ and has been studied in many papers, see for example \cite{arendt,D,D1,kunstVogt,LiskVogt,LiskSobVogt,Sem,Sh, Si80,SV} and references therein.

Various sufficient conditions that guarantee that $L$ has a pure point-spectrum in certain spaces is given for example in \cite{bah,LSW,MPW} and references therein.

%%%%%%%%%%%%%%%%%%%%%%%%%%%%%%%%%%%%%%%%%%%%%%%%%%%%%%
\section{Preliminaries}\label{sec1}

\subsection{Elliptic operators and positive solutions}\label{ssectell}
Let $\Gw$ be a domain in $\R^d$, or more generally, a noncompact connected $C^2$-smooth Riemannian manifold of dimension
$d$.  By a positive function we mean a strictly positive function. We assume that $\nu$ is a positive measure on $\Omega$, satisfying $\dnu=f\,\mathrm{vol}$ with $f$ a positive measurable function; $\mathrm{vol}$ being the volume form of $\Omega$ (which is just the Lebesgue measure in the case of a domain $\Gw$ in $\R^d$). We write $\Omega_1 \Subset \Omega_2$ if $\Omega_2$ is open, $\overline{\Omega_1}$ is
compact and $\overline{\Omega_1} \subset \Omega_2$. Denote by $B(x_0,\delta)$ the open ball of radius $\delta>0$ centered at
$x_0$. Let $\mathbf{1}$ be the constant function on $\Gw$ taking
at any point $x\in \Gw$ the value $1$. For a matrix $A(x)=\big[a^{ij}(x)\big]$ and a vector field $b(x)=  \big(b^j(x)\big)$ we denote
$$
(A(x) \,\xi)^i := \sum_{j=1}^d a^{ij}(x) \,\xi_j,\;\; b(x) \cdot \xi  := \sum_{j=1}^d b^j(x)\, \xi_j , \; \mbox{where } \xi\!=\!(\xi_1,\ldots,\xi_d)\!\in\! \R^d.
$$

We associate to $\Gw$ a fixed {\em exhaustion}
$\{\Omega_{n}\}_{n=1}^{\infty}$, i.e., a sequence of smooth,
relatively compact subdomains such that $\Gw_1\neq \emptyset$,
$\Omega_{k}\Subset \Omega_{k+1}$ and
$\cup_{k=1}^{\infty}\Omega_{k}=\Omega$. For every $k\geq 1$, we
denote $\Gw_{k}^\star =\Gw\setminus \overline{\Gw_k}$.

Let $L$ be a
linear, second-order, elliptic operator defined on $\Gw$.  We assume that the coefficients of $L$ are real, and that in any coordinate system
$(U;x_{1},\ldots,x_{n})$, the operator $L$ is of the divergence form
\begin{equation} \label{div_L}
Lu:=-\mathrm{div} \left(A(x)\nabla u +  u\tilde{b}(x) \right)  +
 b(x)\cdot\nabla u   +c(x)u.
\end{equation}
Here, the minus divergence is the formal adjoint of the gradient with respect to the measure $\nu$.

We assume that for every $x\in\Gw$ the matrix $A(x):=\big[a^{ij}(x)\big]$ is symmetric, and the associated real quadratic form
\be\label{ellip}
 \xi \cdot A(x) \,\xi := \sum_{i,j =1}^d \xi_i a^{ij}(x)\, \xi_j \qquad
 \xi \in \Real ^d
\end{equation}
is positive definite. Moreover, throughout the paper it is assumed that $L$ is locally uniformly elliptic, and the coefficients of $L$ are real valued and locally sufficiently regular in $\Omega$. All our results
hold for example when $L$ is of the form \eqref{div_L}, and $A$ and $f$ are locally H\"{o}lder continuous, $b,\,\tilde{b} \in L^p_{\mathrm{loc}}(\Omega; \mathbb{R}^n,\dx)$, and $c \in L^{p/2}_{\mathrm{loc}}(\Omega; \mathbb{R},\dx)$ for some $p > d$. However it would be apparent from the proofs that any conditions that guarantee standard elliptic regularity theory are sufficient. By a {\em potential} defined in $\Gw$, we mean a function $V \in L^{p/2}_{\mathrm{loc}}(\Omega; \mathbb{R},\!\dx)$ for some $p > d$.

The formal adjoint $L^\star $ of the operator $L$ is defined on its natural space
$L^2(\Omega, \!\!\dnu)$. When $L$ is in divergence form (\ref{div_L}) and $b = \tilde{b}$, the operator
\be
\nonumber
Lu = - \mathrm{div} \left(A \nabla u + u b \right) + b \cdot \nabla u + c u,
\ee
is {\em symmetric} in the space $L^2(\Omega, \dnu)$. Throughout the paper, we call this setting the {\em symmetric case}.  We note that if $L$ is symmetric and $b$ is smooth enough, then $L$ is in fact a Schr\"odinger-type operator of the form
\be
\nonumber
Lu = - \mathrm{div}\big(A \nabla u \big) + V u, \quad \mbox{where } V:=\big(c-\mathrm{div}\, b\big).
\ee

By a solution $v$ of the equation $Lu=0$ in $\Gw$, we mean $v\in W^{1,2}_{\loc}(\Gw)$ that satisfies  the equation $Lu=0$ in $\Gw$ in the weak sense. Subsolutions and supersolutions are defined similarly. We denote the cone of all positive solutions of the
equation $Lu = 0$ in $\Omega$ by $\mathcal{C}_{L}(\Gw)$.

\begin{remark}
We would like to point out that the theory of positive solutions of the equation $Lu=0$ in $\Gw$ (the so-called {\em criticality theory}), and
in particular the results of this paper, are also valid for the
class of classical solutions of locally uniformly elliptic operators of the form
\be \label{L}
Lu=-\sum_{i,j=1}^d a^{ij}(x)\partial_{i}\partial_{j}u + b(x)\cdot\nabla u+c(x) u,
\end{equation}
with real and locally H\"older continuous coefficients,
and for the class of strong solutions of locally uniformly elliptic operators of the form \eqref{L} with locally bounded coefficients (provided that the formal adjoint operator  also satisfies the same assumptions). Nevertheless, for the
sake of clarity, we prefer to present our results only for the
class of weak solutions. \end{remark}

Fix a nonzero nonnegative potential $W$ defined in $\Gw$, and for $\gl\in\R$ denote by $L_\gl$ the elliptic operator $L-\gl W$. Consider
the (weighted) {\em generalized principal eigenvalue}  of the operator $L$
\begin{equation}\label{eq_lambda_0}
    \gl_0= \gl_0(L,W,\Gw):= \sup\{\gl \in \Real\mid \mathcal{C}_{L_\gl}(\Gw)\neq \emptyset\}.
\end{equation}
Note that if $\gl_0\neq -\infty$ (as assumed throughout the present paper), then  in fact, $\gl_0= \max\{\gl \in \Real\mid \mathcal{C}_{L_\gl}(\Gw)\neq \emptyset\}$.

\medskip

If $\gl_0\geq 0$, then for every $k \geq 1$ the Dirichlet Green
function $\Green{\Omega_{k}}{L}{x}{y}$ of the operator $L$ in $\Gw_k$ exists and is positive. By
the generalized maximum principle,
$\{\Green{\Omega_{k}}{L}{x}{y}\}_{k=1}^{\infty}$ is an increasing
sequence converging as $k\to\infty$ either to $\Green{\Omega}{L}{x}{y}$, the
positive {\em minimal  Green function} of $L$ in $\Omega$, and
then $L$ is said to be a {\em subcritical operator} in $\Omega\,
,$ or to infinity and in this case $L$ is  {\em critical} in
$\Omega$. If $\mathcal{C}_{L}(\Omega) = \emptyset$, then $L$ is
{\em supercritical} in $\Gw$  \cite{Mu86,P88a} (cf. \cite{Si80}).

It follows that $L$ is critical (resp.    subcritical) in $\Omega$,
if and only if $L^\star $ is critical (resp.    subcritical) in $\Omega$.
Clearly,  $L_\gl$ is subcritical in $\Gw$ for every $\gl\in
(-\infty,\gl_0)$, and supercritical for $\gl >\gl_0$.
Furthermore,  if $L$ is critical in $\Omega$, then
$\mathcal{C}_{L}(\Omega)$ is a one-dimensional cone and any
positive supersolution of the equation $Lu=0$ in $\Gw$ is a
solution. So, in the critical case,  $\gf \in \mathcal{C}_{L}(\Omega)$ is uniquely defined (up to a multiplicative positive constant), and  such $\gf$ is called
the {\em Agmon ground state of $L$ in $\Gw$}.

Subcriticality is a stable property in the following sense. If $L$
is subcritical in $\Gw$ and $V$ is a potential with a compact support in $\Gw$, then there exists $\varepsilon>0$ such that $L-\gm V$ is
subcritical, for all $|\gm | < \varepsilon$ \cite{Mu86,P88a}. On the other
hand, if $L$ is critical in $\Gw$ and $V$ is a
nonzero, {\em nonnegative} potential, then for any $\varepsilon>0$ the operator
$L+\varepsilon V$ is  subcritical and $L-\varepsilon V$ is  supercritical in
$\Gw$.

\begin{definition}[{\rm Agmon \cite{Ag2}}] \label{defminimalg}
Let $L$ be an elliptic operator of the form \eqref{div_L} defined in  $\Omega$. A function
$u \in \mathcal{C}_L(\Gw^\star_n)$ is said to be a {\em positive
solution of the equation $Lu=0$ of minimal growth in a neighborhood
of infinity in} $\Omega$, if for any  $k> n$ and any $v\in
\overline{C(\Gw^\star _k)}$ which is a positive supersolution of the equation $Lw=0$ in $\Gw^\star _k$, the inequality
$u\le v$ on $\partial \Gw _k$ implies that $u\le v$ in $\Gw _k^\star $.
\end{definition}

In the sequel, in order to simplify our terminology,  we will call a positive  minimal  Green function  - a {\em  Green function}, an Agmon ground state - a {\em ground state}, and a positive
solution of minimal growth in a neighborhood
of infinity in $\Omega$, a {\em positive
solution of minimal growth} in $\Omega$.

\medskip

It turns out that if $L$ is subcritical in $\Gw$, then for any fixed $y\in \Gw$ the Green function $G_L^\Gw(\cdot,y)$ is a positive solution of the equation $Lu=0$ in $\Gw\setminus\{y\}$ of minimal growth in $\Omega$. On the other hand, if $L$ is critical in $\Gw$, then the ground state $\gf$ is a (global) positive solution of the equation $Lu=0$ in $\Gw$ of minimal growth in $\Omega$.

\medskip

Fix a nonzero nonnegative potential $W$ defined in $\Gw$. Let $v$ and $\tilde{v}$ be positive
solutions of the equations $L_{\gm}u=0$ and $L^\star _{\gm}u=0$ in
$\Gw$, respectively, where $\gm\leq \gl_0$. Then \cite{Pcrit2} for
every $\gl<\gm$ we have
  \begin{equation}\begin{split}\label{gseqsc}
  \displaystyle{\int_{\Omega}\Green{\Omega}{L_\gl}{x}{y}W(y)v(y)\dnu(y)}
\leq \displaystyle{\frac{v(x)}{\gm-\gl}} & \qquad\forall x\in \Gw, \\[5mm]
\displaystyle{ \int_{\Omega}\Green{\Omega}{L_\gl}{x}{y}W(x)\tilde{v}(x)\dnu(x)}
\leq \displaystyle{\frac{\tilde{v}(y)}{\gm-\gl}} & \qquad\forall y\in \Gw.
  \end{split}
 \end{equation}
 Moreover, in each of
the inequalities in \eqref{gseqsc} either equality or strict
inequality holds for all points in $\Gw$ and $\gl<\gm$. If equality
holds for $v$ (resp.    $\tilde{v}$), then $v$ (resp.    $\tilde{v}$)
is said to be a positive {\em invariant solution} of the equation
$L_{\gm}u=0$ (resp.    $L^\star _{\gm}u=0$) in $\Gw$, or $\gm$-{\em
invariant solution} of the operator $L$ (resp.    $L^\star $) in $\Gw$.

Assume now that  $L_{\gl_0}$ is critical in $\Gw$ with $\gl_0\in
\mathbb{R}$, and let $\gf$ and $\tilde{\gf}$ be the ground states
of $L_{\gl_0}$ and $L^\star _{\gl_0}$, respectively. Then
\cite[Theorem~2.1]{Pcrit2} $\gf$ and $\tilde{\gf}$ are positive
invariant solutions of the equations $L_{\gl_0}u=0$ and
$L^\star _{\gl_0}u=0$, respectively, Hence, for every $\gl<\gl_0$ we
have \begin{equation}\begin{split}\label{gseq}
\displaystyle{\int_{\Omega}\Green{\Omega}{L_\gl}{x}{y}W(y)\gf(y)\dnu(y)
=\frac{\gf(x)}{\gl_0-\gl}} &  \quad\forall x\in \Gw,\\[5mm]
\displaystyle{\int_{\Omega} \Green{\Omega}{L_\gl}{x}{y}W(x)\tilde{\gf}(x)\dnu(x)
=\!\frac{\tilde{\gf}(y)}{\gl_0-\gl}} & \quad \forall y\in \Gw.
 \end{split}
 \end{equation}

  Assume further that
$W$ is a positive function. If $\gf\,\tilde{\gf} \in
L^1(\Gw,W\dnu)$, then $L_{\gl_0}$ is called {\em positive-critical
in $\Gw$ with respect to $W$}. Otherwise, $L_{\gl_0}$ is called
{\em null-critical in $\Gw$ with respect to $W$}.
\begin{remark}\label{remnoninv}
 Let $\gm\leq \gl_0$, and suppose that $L_{\gm}$ is a subcritical
operator in $\Gw$. Let $v$ and $\tilde{v}$ be positive solutions
of the equations $L_{\gm}u=0$ and $L^\star _{\gm}u=0$ in $\Gw$,
respectively, such that $v\,\tilde{v} \in L^1(\Gw,W\dnu)$. Since for every fixed $x$ (resp.    $y$)
the function $\Green{\Omega}{L_{\gm}}{x}{\cdot}$ (resp.    $\Green{\Omega}{L_{\gm}}{\cdot}{y}$)
is a positive
solution of the operator $L_{\gm}^\star $ (resp.    $L_{\gm}$)  of minimal growth in $\Omega$, the integrability condition $v\,\tilde{v} \in L^1(\Gw,W\dnu)$ implies that
\begin{equation}\label{eq_nonenvar}
\int_{\Omega}\!\!\Green{\Omega}{L_{\gm}}{x}{y}W(y)v(y)\dnu(y)
<\infty, \quad \mbox{and }
\int_{\Omega}\!\!\Green{\Omega}{L_{\gm}}{x}{y}W(x)\tilde{v}(x)\dnu(x)<\infty.
\end{equation}
In light of \cite[Lemma~2.1]{PStroock}, \eqref{eq_nonenvar} implies  that $v$ and
$\tilde{v}$ are not $\gm$-invariant positive solutions of the equations
$L_{\gm}u=0$ and $L^\star _{\gm}u=0$  in $\Gw$, respectively. In other words, if $v\in \mathcal{C}_{L_\gm}(\Gw)$ (resp.    $\tilde{v}\in\mathcal{C}_{L_\gm^\star}(\Gw)$) is a positive $\gm$-invariant solution
of the operator $L$ (resp.    $L^\star$),  and $v\,\tilde{v} \in L^1(\Gw,W\dnu)$, then $\gm=\gl_0$ and $L_{\gl_0}$ is positive-critical with respect to $W$ in $\Gw$.
 \end{remark}
The following example, which is a modification of the
counterexamples to Stroock's conjecture given in \cite{PStroock},
demonstrates that for $\gm=\gl_0$ there exists a subcritical operator $L_{\gl_0}$
and a potential $W>0$ satisfying all the properties of
Remark~\ref{remnoninv}.
\begin{example}\label{exnoninv}
Consider the operator $L:=-\gr\Gd$ on $\mathbb{R}^d$, where $d\geq
3$, and $\gr$ is a strictly positive smooth function. Let
$W:=\mathbf{1}$. Then $L$ is a subcritical operator in
$\mathbb{R}^d$, and it follows from Liouville's theorem that the
functions $v=\mathbf{1}$ and $\tilde{v}=1/\gr$ are (up to a
multiplicative constant) the unique positive solutions of the
equations $Lu=0$ and $L^\star u=0$ in $\mathbb{R}^d$, respectively. We
claim that there exists a smooth positive function $\gr$ so that
$v\tilde{v}=1/\gr\in L^1(\mathbb{R}^d,\dx)$, and
$\gl_0(L,\mathbf{1},\mathbb{R}^d)=0$.

Indeed, let $0<\gb<1$, and $x_k:=(k,0,\ldots,0)$, where, $k=1,2,
\ldots$\,. Finally let $\{\vge_k\}\subset (0,1)$ be a sequence
satisfying $\sum_{k=1}^\infty \vge_k^{d-(2+\gb)} <\infty$.
Take a smooth positive function $\tilde{v}\in L^1(\mathbb{R}^d,\dx)$
satisfying $\tilde{v}(x)\!\!\upharpoonright_{B(x_k,\vge_k)}=(\vge_k)^{-(2+\gb)}$. In particular,
$v\tilde{v}\in L^1(\mathbb{R}^d,\mathbf{1}\!\dx)$.

On the other hand,
we clearly have $\gl_0(L,\mathbf{1},B(x_k,\vge_k))<C \vge_k^{\gb}$, and therefore, $\gl_0(L,\mathbf{1},\mathbb{R}^d)=0$. Moreover, by Remark~\ref{remnoninv}, the unique positive solution
$v$ (resp.    $\tilde{v}$) of the equation
$L_{\gl_0}u=0$ (resp.    $L^\star _{\gl_0}u=0$) is not $\gl_0$-invariant.
\end{example}
\begin{remark}\label{remstroock}  We note that Example~\ref{exnoninv} is in fact a
strengthening of the counterexamples to Stroock's conjecture given
in \cite{PStroock}. It gives an example of a subcritical operator $L$ on
$\mathbb{R}^d$, $d\geq 3$, with $\gl_0=0$, such that the operators
$L$ and $L^\star $  do not admit $\gl_0$-invariant positive solutions,
and in addition, the product of positive entire solutions of the
equations $Lu=0$ and $L^\star u=0$ is in $L^1(\mathbb{R}^d)$. Recall that if a Schr\"odinger-type operator
admits a `small' positive solution $\psi$ in $\Gw$ (and in particular an $L^2$-positive solution), then the operator is critical in $\Gw$, and in particular, $\psi$ is an invariant positive solution \cite{P07}.
 \end{remark}

The following notions of small and semismall perturbations play a
fundamental role in criticality theory \cite{Mu86,Mu97,P88a,P89}. Semismall perturbations revisit in the present paper. It turns out that  they guarantee the compactness of the weighted Green operators in $L^p(\phi_p)$ for all $1\leq p\leq \infty$ (see Section~\ref{secanti}).

\begin{definition} \label{semispertdef1}
  Let $L$ be a subcritical operator in $\Gw$, and let $V$ be a potential.

({\em i}) We say that $V$ is a {\em  small perturbation}
 of $L$ in $\Omega$ if
\be \label{sperteq} \lim_{k\rightarrow \infty}\left\{\sup_{x,y\in
\Omega_{k}^\star }\int_{\Omega_{k}^\star }
\frac{\Green{\Omega}{L}{x}{z}|V(z)|
\Green{\Omega}{L}{z}{y}}{\Green{\Omega}{L}{x}{y}}\dnu(z)\right\}=0\,.
\end{equation}

({\em ii})  We say that $V$ is a {\em  semismall perturbation} of
$L$ in $\Omega$ if for some (all) fixed $x_0\in \Gw$ we have \be \label{semisperteq1} \lim_{k\rightarrow
\infty}\left\{\sup_{y\in \Omega_{k}^\star } \int_{\Omega_{k}^\star }
\frac{\Green{\Omega}{L}{x_0}{z}|V(z)|\Green{\Omega}{L}{z}{y}}
{\Green{\Omega}{L}{x_0}{y}}\dnu(z)\right\}=0\;.
\end{equation}
\end{definition}
\begin{remark}\label{remspert}
({\em i}) A small perturbation of $L$ in $\Omega$ is a semismall perturbation of $L$ and $L^\star$ in $\Omega$ \cite{Mu97}.

\medskip

({\em ii})  We note that $\gl_0$ is well defined by \eqref{eq_lambda_0} even if the potential $W$ does not have a definite sign. It turns out \cite{Mu97, P89} that
if $L$ is subcritical and $W \nleqslant
 0$ is a  semismall (resp.    small) perturbation of $L^\star $ in $\Gw$, then
$\gl_0 > 0$, and $L_{\gl_0}$ is critical in $\Gw$ with a  ground
state $\gf$. Moreover, for each $\gl<\gl_0$ such that the positive Green function $G^{\Omega}_{L_{\gl}}$ exists there exists a positive constant $C_{\gl,x_0,\vge}$ (resp.    $C_\gl$) such that
\begin{equation}
\begin{split}\label{eq_sp}
\!\!\!\!(C_{\gl,x_0,\vge})^{-1}\!\Green{\Omega}{L_{\gl}}{x}{x_0} \!\leq\! \gf(x) \!\leq \!
C_{\gl,x_0,\vge}\!\Green{\Omega}{L_{\gl}}{x}{x_0} &\quad \forall x\!\in\! \Gw, \mathrm{dist}(x,x_0)\!>\!\vge,\\[4mm]
\!\!\!\!\Big(\!\mbox{resp.    }(C_\gl)^{-1}\!\Green{\Omega}{L_{\gl}}{x}{y} \!\leq\! \Green{\Omega}{L}{x}{y} \!\leq\!
C_\gl\!\Green{\Omega}{L_{\gl}}{x}{y} &\quad \forall x,y\!\in\! \Gw, x\!\neq \!y \!\Big).
\end{split}
\end{equation}
Since $L_{\gl_0}$ is critical if and only if $L_{\gl_0}^\star$ is critical, \eqref{gseq} and \eqref{eq_sp} imply that if  $W > 0$ is a  semismall (resp.    small) perturbation of $L^\star $ in $\Gw$, then  $L_{\gl_0}$ is positive-critical. In particular, $\gf$ satisfies \eqref{gseq}.

\medskip

({\em iii}) Murata \cite{Mu07} proved that if $L$ is symmetric and the corresponding (Dirichlet) semigroup generated by $L$ is {\em intrinsically ultracontractive} on $L^2(\Gw)$ \cite{DS},
then $\mathbf{1}$ is a small perturbation of $L$ in $\Gw$. On the other hand, an example of
Ba\~{n}uelos and Davis in \cite{BD} gives us a finite area domain $\Gw\subset \R^2$ such that $\mathbf{1}$ is a
small perturbation of the Laplacian in $\Gw$, but the corresponding semigroup is not
intrinsically ultracontractive.
 \end{remark}
%%%%%%%%%%%%%%%%%%%%%%%%%%%
\subsection{Functional spaces}\label{ssectfunactional} Let $B$ be
a Banach space and $B^\star $ its dual. If $T:B\to B$ is a (bounded) operator
on $B$, we denote by $T^\star $ its dual, and the operator norm of $T$ by
$\|T\|_B$.  The range and the kernel of $T$ are denoted by
$R(T)$ and $N(T)$, respectively. We denote by $\gs(T)$,
$\gs_{\mathrm{point}}(T)$ and $\gr(T)$ the spectrum, the
point-spectrum, and the resolvent set of the operator $T$. If
$\gl\in\gr(T)$, then we denote by $R(\gl,T):=(\gl I-T)^{-1}$
the resolvent of $T$, where I is the identity map on $B$. For every $f \in B$ and $g^\star  \in B^\star $ we
use the notation $\langle g^\star ,f\rangle := g^\star (f)$.
If $T$ acts on two Banach spaces $X$ and $Y$, we distinguish the operators by using the notation $T\!\!\upharpoonright_X$, $T\!\!\upharpoonright_Y$,
respectively.

Let $1\leq p< \infty$,  and let $w$ be a fixed (strictly) positive measurable
weight function defined on $\Gw$. Denote the real ordered Banach
space
$$L^p(w):=L^p(\Gw,w^p\dnu)=\{u\mid\;uw\in L^p(\Gw,\dnu)\}$$ equipped with the
norm $$\|u\|_{p,w}:=\|uw\|_{p}=\left[\int_\Gw
|u(x)\,w(x)|^p\dnu(x)\right]^{1/p}.$$ For $p=\infty$, let
 $$L^\infty(w):=\{u\mid\;uw\in L^\infty(\Gw,\dnu)\}$$ equipped with the norm
 $$\|u\|_{\infty,w}:= \|uw\|_\infty=
\mathrm{ess}\, \sup_{\Gw}(|u|w).$$   The ordering on $L^p(w)$  is
the natural pointwise ordering of functions. For the purpose of
spectral theory, we consider also the canonical complexification
of $L^p(w)$ without changing our notation.

For $1\leq p\leq \infty$, let $p'$ be the usual conjugate exponent
of $p$, so, $\frac{1}{p}+\frac{1}{p'}=1$. It is well-known that for $1\leq p<\infty$,
$(L^p(w))^\star =L^{p'}(w^{-1})$, and in particular, the space $L^p(w)$
is reflexive for all $1<p<\infty$.

Let $W,\gf,\tilde{\gf}$ be positive continuous functions in $\Gw$.
For $1\leq p\leq \infty$, denote
\begin{equation}\label{eq:2.9}
\gf_p:=\gf^{-1}(\gf
W\tilde{\gf})^{1/p},  \qquad \tilde{\gf}_p:=\tilde{\gf}^{-1}(\gf
W\tilde{\gf})^{1/p},
\end{equation}
and consider the corresponding family of weighted Lebesgue spaces $L^p(\gf_p)$, and $L^p(\tilde{\gf}_p)$.

We note that $L^1(\gf_1)$ is independent of $\gf$ while $L^\infty(\gf_\infty)$ is independent of $\tilde{\gf}$ and $W$. Moreover, if $\tilde{\gf}=\gf$ (which is often the case when $L$ is symmetric), then
$L^2(\gf_2)=L^2(\Gw,W\dnu)$ and this space is $\gf$~independent.

It can be easily checked that for $1\leq p< \infty$ we have
\be\label{duality} (L^p(\gf_p))^\star =L^{p'}(\tilde{\gf}_{p'}),
 \ee
where the pairing between
  $L^p(\gf_p)$ and $L^{p'}(\tilde{\gf}_{p'})$  is given by
$$<g^\star ,f>= \int_{\Omega}g^\star (x)W(x)f(x)\dnu(x)
\qquad \forall g^\star \in L^{p'}(\tilde{\gf}_{p'}), f\in L^p(\gf_p).$$
Here the duality is provided by the bilinear rather than the
sesquilinear form\footnote{This is not essential, but simplifies
somewhat the calculations.}.

\medskip

Suppose now that
 \be\label{poscrit}
 \gf W\tilde{\gf}\in\ L^1(\Gw,\dnu),
\qquad \int_\Gw\gf(x) W(x)\tilde{\gf}(x)\dnu(x)=1.
 \ee
 Then by the
H\"older inequality we have the continuous embeddings
 \be\label{imbedding}
 L^\infty(\gf_\infty)\subset L^q(\gf_q) \subset L^p(\gf_p) \subset
L^1(\gf_1), \ee for all $1\leq p\leq q\leq \infty$, and for $f\in
L^\infty(\gf_\infty)$ we have \be\label{norminbedding}
\|f\|_{1,W\tilde{\gf}}=\|f\|_{1,\gf_1} \leq\|f\|_{p,\gf_p} \leq
\|f\|_{q,\gf_q} \leq \|f\|_{\infty,\gf_\infty}=
\|f\|_{\infty,\gf^{-1}}.
 \ee
Moreover, $\|f\|_{1,\gf_1} = \|f\|_{\infty,\gf_\infty}=1$ if and only if $|f|=\gf=1$ almost everywhere.
In particular, $\gf\in L^p(\gf_p)$ for every $1\leq p \leq \infty$, and \eqref{poscrit} implies that
 \be\label{normgf}
\|\gf\|_{1,\gf_1}
=\|\gf\|_{p,\gf_p}=\|\gf\|_{\infty,\gf_\infty}=1 \qquad \forall\, 1\leq p\leq \infty, \ee
so, the norms of the embeddings in \eqref{imbedding} is $1$.
Moreover, these embeddings are dense. We also note that if  $\tilde{\gf}=\gf$ (as in the symmetric case), then
\be\label{simbedding} (L^1(W\gf))^\star =L^\infty(\gf^{-1})\subset
L^2(\Gw, W\dnu) \subset L^1(W\gf). \ee
%%%%%%%%%%%%%%%%%%%%%%%%%%%
\begin{remark}\label{rem_unique}
Throughout the paper we fix an operator $L$ of the form \eqref{div_L}, a positive potential $W$, $\gm\leq\gl_0$, and $\gf$, $\tilde{\gf}$ two positive solutions of the
equations $L_\gm u=0$ and $L^\star_\gm u=0$ in $\Gw$, respectively.  We study properties of a family of the corresponding weighted Green operators on $L^p(\gf_p)$ and $L^p(\tilde{\gf}_p)$. We note that if $\gm=\gl_0$ and $L_{\gl_0}$ is critical in $\Gw$, then the spaces $L^p(\gf_p)$ and $L^p(\tilde{\gf}_p)$ are uniquely defined.
 \end{remark}
\begin{remark}\label{rem_htrans}
Let $\gf$ and $\tilde{\gf}$ be two fixed positive solutions of the
equations $L u=0$ and $L^\star u=0$ in $\Gw$, respectively.
For $\gm\leq \gl_0$, define the operator
$$L_\gm^\gf:= \frac{1}{\gf}L_\gm\gf= \frac{1}{\gf}L\gf -\gm W=L^\gf-\gm W,$$
which is called {\em Doob's $\gf$-transform} (or the {\em ground state transform with respect to} $\gf$) of the operator $L_\gm$.
Note that for $\gm\leq \gl_0$ the operator ${L_\gm}$ is subcritical in $\Gw$ if and only if $L^\gf_\gm$ is subcritical in $\Gw$, and we have
$$\Green{\Omega}{L^\gf_\gm}{x}{y}=\frac{1}{\gf(x)}\Green{\Omega}{L_\gm}{x}{y}\gf(y).$$
Clearly, $L^\gf \mathbf{1} =0$ and $(L^\gf)^\star  (\gf \tilde{\gf})=0$. In particular, $L^\gf$ is a {\em diffusion operator}.
We note that for $1\leq p\leq \infty$, the weighted $L^p$-spaces associated with the positive solutions $\mathbf{1}$
and $\gf\tilde{\gf}$ of the equations $L^\gf u=0$ and
$(L^\gf)^\star u=0$, respectively, are just $L^p(\Gw,\gf W \tilde{\gf}\dnu)$. So, in this case (which corresponds to the class of diffusion operators) the corresponding one-parameter weights are $p$-independent.
 \end{remark}
%%%%%%%%%%%%%%%%%%%%%%%%%%%%%%%%%%%%%%%%%%%%%%%%%%%%%%%%%%%%%
\section{Boundedness of the Green operators}\label{secgreen}
%%%%%%%%%%%%%%%%%%%%%%%%%%%%%%%%%%%%%%%%%%%%%%%%%%%%%%%
Fix a positive potential $W$ and  $\gm\leq \gl_0$. Let $\gf$ and $\tilde{\gf}$ be two fixed positive solutions of the equations $L_{\gm}u=0$ and $L^\star _{\gm}u=0$ in $\Gw$,
respectively. For $1\leq p\leq \infty$ let  $\gf_p$ and $\tilde{\gf}_p$ be the functions defined in \eqref{eq:2.9}.
Note that we do not assume below neither that
$\gf$ and $\tilde{\gf}$ are invariant solutions nor that the integrability condition \eqref{poscrit} is
satisfied.

For $\gl<\gm$, we introduce the integral operators
$$\mathcal{G}_\gl f(x) \!:= \!\!\! \int_{\Omega}\!\!\!  \Green{\Omega}{L_\gl
}{x}{y}W(y)f(y)\!\dnu(y),\;\; \mathcal{G}_\gl^\odot f(y) \!:=\!\!\!
 \int_{\Omega}\!\!\!  \Green{\Omega}{L_\gl}{x}{y}W(x)f(x)\!\dnu(x).$$
In the present section we study for $1\leq p\leq \infty$ the boundedness of the {\em weighted Green operators} $\mathcal{G}_\gl$ and $\mathcal{G}_\gl^\odot$ on
$L^p(\gf_p)$ and $L^p(\tilde{\gf}_p)$, respectively. We have
%%%%%%%%%%%%%%%%
\begin{theorem}\label{thmbddgen0}
Let $L$ be an elliptic operator on $\Gw$ of the form \eqref{div_L},
and let $W$ be a positive potential. Fix  $\gm\leq
\gl_0$, and let $\gf$ and $\tilde{\gf}$ be two fixed positive
solutions of the equations $L_{\gm}u=0$ and $L^\star _{\gm}u=0$ in
$\Gw$, respectively.  Then

\begin{enumerate}
\item For $1\leq p\leq\infty$, the operator
$\mathcal{G}_\gl\!\!\upharpoonright_{L^p(\gf_p)}$ (resp.,
$\mathcal{G}_\gl^\odot\!\!\upharpoonright_{L^p(\tilde{\gf}_p)}$) is a well defined bounded and
positive improving operator on $L^p(\gf_p)$ (resp.    $L^p(\tilde{\gf}_p)$).
 Moreover, we have \be \label{gnormgen}
\|\mathcal{G}_\gl\|_{L^p(\gf_p)}\leq (\gm-\gl)^{-1}, \qquad
\mbox{(resp.    } \|\mathcal{G}_\gl^\odot\|_{L^p(\tilde{\gf}_p)}\leq
(\gm-\gl)^{-1} \mbox{ )}.
 \ee

\item For $1\leq p< \infty$, the operator
$\mathcal{G}_\gl^\odot\!\!\upharpoonright_{L^{p'}(\tilde{\gf}_{p'})}$ is the dual
operator of $\mathcal{G}_\gl\!\!\upharpoonright_{L^{p}(\gf_{p})}$, and
$\mathcal{G}_\gl\!\!\upharpoonright_{L^{p'}(\gf_{p'})}$ is the dual
 of $\mathcal{G}_\gl^\odot\!\!\upharpoonright_{L^{p}(\tilde{\gf}_{p})}$.
\item Suppose that  $\gf$ is a $\gm$-invariant positive solution of the operator $L$, and $p=\infty$, then $\|\mathcal{G}_\gl\|_{L^\infty(\gf_\infty)}= (\gm-\gl)^{-1}$.

\item Suppose that  $\gf$ is a $\gm$-invariant positive solution  of the operator $L$ satisfying \eqref{poscrit}, then  for any $1\leq p\leq \infty$, $\|\mathcal{G}_\gl\|_{L^p(\gf_p)}= (\gm-\gl)^{-1}$.
\end{enumerate}
\end{theorem}
\begin{proof}
(1)  Let $f\in L^\infty(\gf_\infty)$, then by \eqref{gseqsc}
\begin{multline}\label{gwfsc1}
  |\mathcal{G}_\gl f(x)|\leq \int_{\Omega}\Green{\Omega}{L_\gl}{x}{y}W(y)|f(y)|\dnu(y)\leq\\
\|f\|_{\infty,\gf_\infty}\int_{\Omega}\Green{\Omega}{L_\gl}{x}{y}W(y)\gf(y)\dnu(y)
\leq\frac{\|f\|_{\infty,\gf_\infty}}{\gm-\gl}\gf(x),
\end{multline}
so,  $\|\mathcal{G}_\gl\|_{L^\infty(\gf_\infty)}\leq
(\gm-\gl)^{-1}$. Similarly,
$\|\mathcal{G}^\odot_\gl\|_{L^\infty(\tilde{\gf}_\infty)}\leq
(\gm-\gl)^{-1}$.

Assume now that $f\in L^1(\gf_1)$, then by the Tonelli-Fubini theorem and \eqref{gseqsc} we obtain
\begin{multline}\label{gwf7}
  \|\mathcal{G}_\gl f(x)\|_{1,\gf_1}=
  \int_\Gw W(x)\tilde{\gf}(x)\left|\int_{\Omega}\Green{\Omega}{L_\gl}{x}{y}W(y)f(y)
  \dnu(y)\right|\dnu(x)\leq\\[2mm]
\int_\Gw W(x)\tilde{\gf}(x)\int_{\Omega}\Green{\Omega}{L_\gl}{x}{y}W(y)|f(y)|
  \dnu(y)\dnu(x) = \\[2mm]
\int_{\Omega}\left(\int_\Gw
W(x)\tilde{\gf}(x)\Green{\Omega}{L_\gl}{x}{y}\dnu(x)\right) W(y)|f(y)|
  \dnu(y)\leq\\[2mm]
  \frac{1}{\gm-\gl}\int_{\Omega}\tilde{\gf}(y) W(y)|f(y)|
  \dnu(y)=\frac{\|f\|_{1,\gf_1}}{\gm-\gl}\,.
\end{multline}
Hence, $\|\mathcal{G}_\gl\|_{L^1(\gf_1)}\leq (\gm-\gl)^{-1}$. Similarly,
$\|\mathcal{G}^\odot_\gl\|_{L^1(\tilde{\gf}_1)}\leq
(\gm-\gl)^{-1}$.

For $1<p<\infty$, the boundedness of
$\mathcal{G}_\gl\!\!\upharpoonright_{L^p(\gf_p)}$ with norm estimate
$\|\mathcal{G}_\gl\|_{L^p(\gf_p)}\leq (\gm-\gl)^{-1}$ follows now
directly from a Riesz-Thorin-type interpolation theorem with
weights proved by Stein \cite[Theorem~2]{St}.

 \vspace{.1cm}

(2) The duality claim follows now directly from
\eqref{duality}.

\vspace{.1cm}

(3) and (4) follow from (1) and \eqref{gseq}.
\end{proof}
\begin{remark}\label{rem_schurtest1}
Theorem~\ref{thmbddgen0} (and \eqref{gnormgeneqp}) for $1<p<\infty$ follows also from the Schur test with weights \cite[Lemma~5.1]{KVZ}.
Indeed, set $$K(x,y):= \frac{G(x,y)}{\phi^{1-p}(y)\tilde{\phi}(y)}\,,\quad w(x,y):=\frac{\phi^{p}(y)}{\phi^{p}(x)}\,,\quad \drho(y):=\left(\phi_{p}(y)\right)^{p}\dnu(y).$$
Then \eqref{gseqsc} implies that
\begin{equation}\begin{split}\label{gseqscp}
 \displaystyle{\int_{\Omega}\!\!w(x,y)^{\frac{1}{p}}K(x,y)\drho(y)}\!=\!
 \frac{\int_{\Omega}\Green{\Omega}{L_\gl}{x}{y}W(y)\phi(y)\dnu(y)}{\phi(x)}
\!\leq\! \displaystyle{\frac{1}{\gm-\gl}} & \quad\forall x\in \Gw, \\[5mm]
\displaystyle{ \int_{\Omega}\!\!w(x,y)^{-\frac{1}{p'}}K(x,y)\drho(x)}\!=\!
\frac{ \int_{\Omega}\Green{\Omega}{L_\gl}{x}{y}W(x)\tilde{\phi}(x)\dnu(x)}{\tilde{\phi}(y)}
\!\leq\! \displaystyle{\frac{1}{\gm-\gl}} & \quad\forall y\in \Gw.
  \end{split}
 \end{equation}
Applying the aforementioned Schur test we get $\|\mathcal{G}_\gl\|_{L^p(\gf_p)}\leq (\gm-\gl)^{-1}$.

The Schur test with weights is essentially a theorem of Aronszajn, and in fact follows from Stein's Riesz-Thorin-type interpolation theorem with
weights \cite[Theorem~2]{St}.
\end{remark}

\begin{remark}\label{rem_pos_crt}
It follows from part (ii) of Theorem~\ref{thmbddgen1} that the assumptions of part (4) of Theorem~\ref{thmbddgen0} imply that in fact $\gm=\gl_0$ and $L_{\gl_0}$ is positive-critical in $\Gw$ with respect to $W$.
\end{remark}
%%%%%%%%%%%
\begin{remark}\label{rempindepnd}
The norm estimate \eqref{gnormgen} does not depend on  $\gf$, $\tilde{\gf}$ and $W$ and $p$.
\end{remark}
\begin{remark}\label{rem_nonegative}
The requirement that $W$ is strictly positive can be weakened, and
Theorem~\ref{thmbddgen0}  holds in a slightly weaker sense if $W$
is a nonzero nonnegative function. Indeed, let $1\leq p\leq
\infty$. Since Stein's Riesz-Thorin-type interpolation theorem with
weights \cite[Theorem~2]{St} holds for nonnegative weights,
we have for $\gl<\gm$
 \be\label{wgeq0} \|(\mathcal{G}_\gl f)
\gf_p\|_{L^p(\Gw,\dnu)}\leq \frac{1}{\gm-\gl}\|f
\gf_p\|_{L^p(\Gw,\dnu)}\qquad \forall f \mbox{ s.t. } f\gf_p\in
L^p(\Gw,\dnu). \ee
 \end{remark}
%%%%%%%%%%%%%
\section{Principal eigenfunction}\label{sec_principal}
The  Krein-Rutman theorem roughly asserts that if $T$ is a compact operator defined on a Banach space $X$ with a total cone $P$ such that $T$ is positive improving and its spectral radius $r(T)$ is strictly positive, then $T$  admits a positive eigenfunction with an eigenvalue $r(T)$. Moreover, under an irreducibility assumption $r(T)$ is simple.  The weighted Green operator $\mathcal{G}_\gl$ in the weighted Lebesgue spaces $L^p(\gf_p)$ is positive improving but in general, $\mathcal{G}_\gl$ is not compact. Nevertheless, under some further conditions it admits a positive eigenfunction with an eigenvalue $(\gm-\gl)^{-1}$.

Throughout the present section, as in Section~\ref{secgreen}, $W$ is a fixed positive potential, $\gm\leq \gl_0$, and $\gf$, $\tilde{\gf}$ are fixed positive solutions
of the equations $L_{\gm}u=0$ and $L^\star _{\gm}u=0$ in $\Gw$, respectively. We study eigenvalues and eigenfunctions of the weighted Green operators  $\mathcal{G}_\gl$.
%%%%%%%%%%%%%
\begin{remark}\label{rem9}
Eigenfunctions of $\mathcal{G}_\gl\!\!\upharpoonright_{L^1(\gf_1)}$ might be not smooth enough to solve weakly the corresponding partial differential equation. Therefore, if $p=1$, we always {\em assume} that such eigenfunctions are also in $ L^q_{\mathrm{loc}}(\Gw)$ for some $q>1$.
 \end{remark}
%%%%%%%%%%%%%%%%%%%%%%%%%%%%%%%%%%%%
We  have
\begin{theorem}\label{thmbddgen}
Let $W$, $\gm$, $\gf$, and $\tilde{\gf}$ be as above, and let $\gl<\gm$. Then
for any $1\leq  p\leq \infty$, zero is not an eigenvalue of the
operators $\mathcal{G}_\gl\!\!\upharpoonright_{L^p(\gf_p)}$ and
$\mathcal{G}_\gl^\odot\!\!\upharpoonright_{L^p(\tilde{\gf}_p)}$.

\vspace{.3cm}

\noindent Moreover, any eigenfunction $\vgf$ (resp.    $\tilde{\vgf}$) of
$\mathcal{G}_\gl\!\!\upharpoonright_{L^p(\gf_p)}$  (resp.
$\mathcal{G}_\gl^\odot\!\!\upharpoonright_{L^p(\tilde{\gf}_p)}$) with an eigenvalue
$\gn$ solves the equation $$(L-[\gl +(\gn)^{-1}]W)\vgf=0 \qquad
\mbox{(resp.    } (L^\star -[\gl +(\gn)^{-1}]W)\tilde{\vgf}=0\mbox{)}
\quad \mbox{in } \Gw.$$
\end{theorem}
\begin{proof}
Let $1\leq p\leq \infty$ and let $\vgf\in L^p(\gf_p)$,
$\|\vgf\|_{p,\gf_p}=1$ be an eigenfunction of the operator
$\mathcal{G}_\gl\!\!\upharpoonright_{L^p(\gf_p)}$ with an eigenvalue $\gn$, and
define
$$\vgf_k(x):=\int_{\Gw_k}
\Green{\Omega_{k}}{L_\gl}{x}{y}W(y)\vgf(y)\dnu(y)\qquad x\in \Gw_k, \;k\geq 1.$$ Clearly,
$$\Green{\Omega_{k}}{L_\gl}{x}{y}W(y)|\vgf(y)|\leq
\Green{\Omega}{L_\gl}{x}{y}W(y)|\vgf(y)|\qquad \mbox{in } \Gw_k.$$
 On the other hand, by Theorem~\ref{thmbddgen0} we have
$\Green{\Omega}{L_\gl}{x}{\cdot}W\vgf \in L^1(\Gw,\dnu)$ for almost
every $x\in \Gw$. Therefore, $\vgf_k(x)$ is well-defined almost
everywhere in $\Gw_k$, and Lebesgue's dominated convergence theorem implies
that
$$\lim_{k\to\infty}\vgf_k(x)=\int_{\Gw}
\Green{\Omega}{L_\gl}{x}{y}W(y)\vgf(y)\dnu(y)=\gn\vgf(x)$$ almost everywhere in $\Gw$.

Since $|\vgf|\in L^p(\gf_p)$, Theorem~\ref{thmbddgen0} implies that
$\mathcal{G}_\gl|\vgf|\in L^p(\gf_p)$.  Obviously,  $|\vgf_k|\leq
\mathcal{G}_\gl|\vgf|\in L^p(\gf_p)$. Consequently, $\{\vgf_k\}$
is bounded in $L^p(\gf_p)$. Note that for $1\leq p<\infty$, Lebesgue's
dominated convergence theorem implies that
$\|\vgf_k\|_{L^p(\gf_p)} \to |\gn|\|\vgf\|_{L^p(\gf_p)}$, and this holds true also for $p=\infty$.

On the other hand,  taking into account Remark~\ref{rem9} in case $p=1$, it follows that each $\vgf_k$ solves the equation
$$(L-\gl W)\vgf_k=W\vgf \qquad  \mbox{in } \Gw_k.$$
A standard elliptic regularity argument implies that
 $\gn\vgf$ solves the equation
$$(L-\gl W)\gn\vgf=W\vgf\neq 0\qquad  \mbox{in } \Gw.$$
In particular, $\gn\neq
0$. Thus, $\vgf$ solves the equation $$(L-[\gl+(\gn)^{-1}]W)\vgf=0
\qquad  \mbox{in } \Gw.$$
\end{proof}
%%%%%%%%%%%%%%%%%%%%%%%
\begin{remark}\label{rem_c}
It was proved in \cite{P99} that zero is not an eigenvalue of
$\mathcal{G}_\gl\!\!\upharpoonright_{C(\gf_\infty)}$.
 \end{remark}
%%%%%%%%%%%%%%%%%%%%%%
The next result concerns conditions under which the positive solution $\gf$ is an eigenfunction of $\mathcal{G}_\gl$ in $L^p(\gf_p)$.
\begin{theorem}\label{thmbddgen1}
Let $W$, $\gm$, $\gf$, and $\tilde{\gf}$ be as above, and let $\gl<\gm$. Then

\vspace{.3cm}

(i) The function $\gf$ (resp.    $\tilde{\gf}$) is a nonnegative
eigenfunction of the operator
$\mathcal{G}_\gl\!\!\upharpoonright_{L^\infty(\gf_\infty)}$ (resp.,
$\mathcal{G}_\gl^\odot\!\!\upharpoonright_{L^\infty(\tilde{\gf}_\infty)}$) with an
eigenvalue $(\gm-\gl)^{-1}$ if and only if  $\gf$ (resp.
$\tilde{\gf}$) is a $\gm$-invariant positive solution with respect
to $L$ (resp.    $L^\star $). In this case
 \begin{multline} \label{gnormgeneq}
\|\mathcal{G}_\gl\|_{L^\infty(\gf_\infty)}=\|\mathcal{G}_\gl^\odot\|_{L^1(\tilde{\gf}_1)}
=(\gm-\gl)^{-1}, \\[4mm]
\mbox{ (resp.    }
\|\mathcal{G}_\gl^\odot\|_{L^\infty(\tilde{\gf}_\infty)}=\|\mathcal{G}_\gl\|_{L^1(\gf_1)}
=(\gm-\gl)^{-1} \mbox{ )}. \end{multline}
 Furthermore, if $\gf$
and $\tilde{\gf}$ are both $\gm$-invariant positive solutions,
then for $1\leq p\leq\infty$ \be \label{gnormgeneqp}
\|\mathcal{G}_\gl\|_{L^p(\gf_p)}=
\|\mathcal{G}_\gl^\odot\|_{L^p(\tilde{\gf}_p)} =(\gm-\gl)^{-1}.
 \ee
 
(ii) Let $1\leq p< \infty$. Then $\gf$ (resp.    $\tilde{\gf}$) is a
nonnegative  eigenfunction of the operator
$\mathcal{G}_\gl\!\!\upharpoonright_{L^p(\gf_p)}$ (resp.,
$\mathcal{G}_\gl^\odot\!\!\upharpoonright_{L^p(\tilde{\gf}_p)}$) with an eigenvalue
$\gn=(\gm-\gl)^{-1}$ if and only if  $\gm=\gl_0$, and the operator
$L_{\gl_0}$ is positive-critical with respect to $W$. In this case,
$\gf$ and $\tilde{\gf}$ are the ground states of $L_{\gl_0}$ and $L_{\gl_0}^\star $, respectively,
and \be \label{gnormgeneqp9}
\|\mathcal{G}_\gl\|_{L^p(\gf_p)}=
\|\mathcal{G}_\gl^\odot\|_{L^p(\tilde{\gf}_p)} =\frac{1}{\gl_0-\gl} \qquad \forall \;1\leq p\leq \infty.
 \ee
\end{theorem}
\begin{proof}
({\em i}) All the claims of this part can be checked easily and
left to the reader. In particular, use Theorem~\ref{thmbddgen0} and Stein's
Riesz-Thorin-type interpolation theorem to prove \eqref{gnormgeneqp}.

 \vspace{.3cm}

({\em ii}) Let $1\leq p< \infty$. The positive solution $\gf$ is an eigenfunction of the
operator $\mathcal{G}_\gl\!\!\upharpoonright_{L^p(\gf_p)}$ with an eigenvalue
$(\gm-\gl)^{-1}$ if and only if $\gf W \tilde{\gf} \in L^1(\Gw)$ and $\gf$ is $\gm$-invariant positive solution.

In particular, if $\gm=\gl_0$, and  $L_{\gl_0}$ is positive-critical with respect to $W$, then $\gf$ is an eigenfunction of the
operator $\mathcal{G}_\gl\!\!\upharpoonright_{L^p(\gf_p)}$ with an eigenvalue $(\gl_0-\gl)^{-1}$.

On the other hand, if $\gf$ is an eigenfunction of the
operator $\mathcal{G}_\gl\!\!\upharpoonright_{L^p(\gf_p)}$ with an eigenvalue $(\gm-\gl)^{-1}$, then $\gf W \tilde{\gf} \in L^1(\Gw)$.

Assume that $L_\gm$ is subcritical in $\Gw$, then
$$\int_{\Omega}\Green{\Omega}{L_{\gm}}{x}{y}W(y)\gf(y)\dnu(y) <\infty. $$
By Remark~\ref{remnoninv}, $\gf$ is not a $\gm$-invariant solution, and we get a contradiction. Therefore, $L_{\gm}$ is critical in $\Gw$ and hence, $\gm= \gl_0$. Since $\gf W \tilde{\gf} \in L^1(\Gw)$, it follows that $L_{\gl_0}$ is positive-critical with respect to $W$.
\end{proof}
%%%%%%%%%%%%%%%%%%%%%%%%%%%%%%%%%%%%%%
In the critical case we have the following result.
\begin{theorem}\label{thmbdd}
Let $L$ be an elliptic operator on $\Gw$ of the form \eqref{div_L},
and let $W$ be a positive potential. Assume that the
operator $L_{\gl_0}$ is critical, and let $\gf$ and $\tilde{\gf}$
be the ground states of $L_{\gl_0}$ and $L^\star _{\gl_0}$,
respectively. Fix $\gl<\gl_0$. Then

\vspace{.3cm}

(i)  For $1\leq p\leq\infty$, we have \be \label{gnorm}
\|\mathcal{G}_\gl\|_{L^p(\gf_p)}=
\|\mathcal{G}_\gl^\odot\|_{L^p(\tilde{\gf}_p)}=
\frac{1}{\gl_0-\gl}\;. \ee

\vspace{.3cm}

(ii) The operator $\mathcal{G}_\gl\!\!\upharpoonright_{L^\infty(\gf_\infty)}$ (resp.,
$\mathcal{G}_\gl^\odot\!\!\upharpoonright_{L^\infty(\tilde{\gf}_\infty)}$) admits a
unique eigenvalue  $\gn=(\gl_0-\gl)^{-1}$  with a nonnegative
eigenfunction. Moreover, $(\gl_0-\gl)^{-1}$ is a simple eigenvalue
of $\mathcal{G}_\gl\!\!\upharpoonright_{L^\infty(\gf_\infty)}$ (resp.,
$\mathcal{G}_\gl^\odot\!\!\upharpoonright_{L^\infty(\tilde{\gf}_\infty)}$). The corresponding eigenfunction is $\gf$ (resp.,
$\tilde{\gf}$), and $\gf$ (resp.    $\tilde{\gf}$) is the unique $L^\infty(\gf_\infty)$ (resp.    $L^\infty(\tilde{\gf}_\infty)$) solution of the equation
 $L_{\gl_0}u=0$ in $\Gw$.
\vspace{.3cm}

(iii) Suppose further that  the operator $L_{\gl_0}$ is positive-critical
with respect to $W$. Then for all $1\leq  p<\infty$ the function $\gf$
(resp.    $\tilde{\gf}$) is the unique (up to a multiplicative constant) nonnegative eigenfunction of
the operator $\mathcal{G}_\gl\!\!\upharpoonright_{L^p(\gf_p)}$ (resp.,
$\mathcal{G}_\gl^\odot\!\!\upharpoonright_{L^p(\tilde{\gf}_p)}$).
%%%%%%%%%%%
\end{theorem}
\begin{proof} ({\em i--ii}) Since $\gf$ is a ground state, it is a $\gl_0$-invariant positive solution with respect
to the operator $L$ and the weight $W$, part ({\em i}) and the existence assertion of ({\em ii}) follow
from part ({\em i}) of Theorem~\ref{thmbddgen1}.

It remains to prove the uniqueness and simplicity of
$(\gl_0-\gl)^{-1}$ for the operator $\mathcal{G}_\gl\!\!\upharpoonright_{L^\infty(\gf_\infty)}$. Let $\vgf$ be a nonnegative eigenfunction of
the operator $\mathcal{G}_\gl\!\!\upharpoonright_{L^\infty(\gf_\infty)}$ with an
eigenvalue $\gn=(\gk-\gl)^{-1}$. Without loss of generality we may
assume that $\|\vgf\|_{\infty,\gf_\infty}=1$, thus $\gf-\vgf\geq 0$. By
Theorem~\ref{thmbddgen}, $\vgf$ is a positive solution of the
equation $(L-\gk W)u=0$ in $\Gw$. Hence, $\gk\leq \gl_0$.
Therefore, $v:=\gf-\vgf$ is a nonnegative supersolution of the
equation $(L-\gl_0 W)u=$ in $\Gw$. On the other hand, $L_{\gl_0}$
is critical in $\Gw$ if and only if  $\gf$ is the unique (up to a
multiplicative constant) nonzero nonnegative supersolution. Thus, $\gk=
\gl_0$ and $\vgf=\gf$. Hence, for all $\gl<\gl_0$,  $\gf$  is the unique nonnegative
eigenfunction of $\mathcal{G}_\gl\!\!\upharpoonright_{L^\infty(\gf_\infty)}$.

Moreover, if $(\gl_0-\gl)^{-1}$ is an eigenvalue of
$\mathcal{G}_\gl\!\!\upharpoonright_{L^\infty(\gf_\infty)}$ with a normalized
eigenfunction $\vgf$, then $u:=\gf-\vgf\geq 0$, and by Theorem~\ref{thmbddgen} $u$ is a nonnegative solution
of the equation $L_{\gl_0}u=0$ in $\Gw$. Since $L_{\gl_0}$ is
critical, it follows that $u = c\gf$ for some $c\geq 0$. So,
$\vgf= (1-c)\gf$, and either $c=0$ or $c=2$. Hence, $(\gl_0-\gl)^{-1}$ is a simple
eigenvalue of $\mathcal{G}_\gl\!\!\upharpoonright_{L^\infty(\gf_\infty)}$.

\vspace{.3cm}

({\em iii}) For $1\leq p< \infty$, part ({\em ii}) of Theorem~\ref{thmbddgen1} implies that $\gf$
is a positive eigenfunction of $\mathcal{G}_\gl\!\!\upharpoonright_{L^p(\gf_p)}$ with an eigenvalue $(\gl_0-\gl)^{-1}$.

%%%%%%%%%%%%%%%%%%%%%%%
Let $\vgf\in L^p(\gf_p)$ be a nonnegative normalized eigenfunction
of the operator $\mathcal{G}_\gl\!\!\upharpoonright_{L^p(\gf_p)}$ with an eigenvalue
$\gn$. Note that $\vgf$ is strictly positive since $\mathcal{G}_\gl\!\!\upharpoonright_{L^p(\gf_p)}$ is positivity improving. Clearly, $\gn\geq 0$, and by Theorem~\ref{thmbddgen} $\gn\neq 0$. Therefore, $\gn$ can be written as  $\gn =(\gk-\gl)^{-1}$, where $\gl<\gk$.

On the other hand,  $\gn\leq \|\mathcal{G}_\gl\|_{L^p(\gf_p)}= (\gl_0-\gl)^{-1}$, and hence
$\gl_0\leq \gk$. By Theorem~\ref{thmbddgen}, $\vgf$ is a nonnegative solution of the equation $L_\gk u=0$ in $\Gw$, therefore $\gk\leq \gl_0$, Thus, $\gk= \gl_0$, and $\vgf=\gf$.
%%%%%%%%%%%%%%%%%%%%%%%%%%%%%%
\end{proof}
The next result deals with the case $p=2$. In this case, any $L^2$-eigen-function of $\mathcal{G}_\gl\!\!\upharpoonright_{L^2(\gf_2)}$ with the `maximal' eigenvalue has a definite sign and this eigenvalue is simple.
\begin{theorem}\label{thmbddpleq2}
Let $L$ be an elliptic operator on $\Gw$ of the form \eqref{div_L},
and let $W$ be a positive potential. Fix $\gl<\gm\leq\gl_0$, and let $\gf$ and $\tilde{\gf}$
be two positive solutions of the equation $L_\gm u=0$ and $L_\gm^\star u=0$ in $\Gw$,
respectively.

If $\vgf$  is an eigenfunction of $\mathcal{G}_\gl\!\!\upharpoonright_{L^2(\gf_2)}$ with an eigenvalue $(\gm-\gl)^{-1}$, then $\vgf$ has a definite sign, and $(\gm-\gl)^{-1}$ is a simple eigenvalue of $\mathcal{G}_\gl\!\!\upharpoonright_{L^2(\gf_2)}$.

\vspace{.3cm}

Assume further that $\gf W \tilde{\gf}\in L^1(\Gw)$. Then $\vgf=c\gf$ for some constant $c$,  $\gm=\gl_0$, and the operator $L_{\gl_0}$ is positive-critical.
Moreover,  $(\gl_0-\gl)^{-1}$ is the unique eigenvalue of $\mathcal{G}_\gl\!\!\upharpoonright_{L^p(\gf_p)}$ with a nonnegative eigenfunction for all $1 \leq p \leq \infty$  and all $\gl<\gl_0$.
Furthermore,  $(\gl_0-\gl)^{-1}$ is  a simple eigenvalue of $\mathcal{G}_\gl\!\!\upharpoonright_{L^p(\gf_p)}$ for all $2\leq p\leq \infty$, and all $\gl<\gl_0$.
\end{theorem}
\begin{proof}
Recall that
$$
\|\mathcal{G}_\gl\|_{L^2(\gf_2)}\leq
\frac{1}{\gm-\gl}. $$
Let $\vgf\in L^2(\gf_2)$ be an eigenfunction
of the operator $\mathcal{G}_\gl\!\!\upharpoonright_{L^2(\gf_2)}$ with an eigenvalue
$(\gm-\gl)^{-1}$.  Thus, $
\|\mathcal{G}_\gl\|_{L^2(\gf_2)}=
(\gm-\gl)^{-1}$.
Without loss of generality, we may assume that $\vgf$ is a real function.  Therefore, due to the positivity improving of $\mathcal{G}_\gl\!\!\upharpoonright_{L^2(\gf_2)}$, and the Cauchy-Schwarz inequality, we obtain
\begin{multline}\label{eq:p=2}
 \frac{1}{\gm-\gl}
 \|\vgf\|_{L^2(\gf_2)}^2=(\vgf,\mathcal{G}_\gl\!\!\upharpoonright_{L^2(\gf_2)}\vgf)_{L^2(\gf_2)} \leq\\[3mm]
 \left(|\vgf|,\Big|\mathcal{G}_\gl\!\!\upharpoonright_{L^2(\gf_2)}\vgf\Big|\right)_{L^2(\gf_2)} \leq
  \left(|\vgf|,\mathcal{G}_\gl\!\!\upharpoonright_{L^2(\gf_2)}\big|\vgf\big|\right)_{L^2(\gf_2)} \leq\\[3mm]
  \|\vgf\|_{L^2(\gf_2)}\|\mathcal{G}_\gl\!\!\upharpoonright_{L^2(\gf_2)}|\vgf|\|_{L^2(\gf_2)}\leq \frac{1}{\gm-\gl}\|\vgf\|_{L^2(\gf_2)}^2.
 \end{multline}
 As a result we have equality signs in all the inequalities of \eqref{eq:p=2}. The equality in the Cauchy-Schwarz inequality implies that
 $$\mathcal{G}_\gl\!\!\upharpoonright_{L^2(\gf_2)}|\vgf|= \frac{1}{\gm-\gl}|\vgf|,$$
 and therefore, $|\vgf|$ is a nonnegative eigenfunction of the operator $\mathcal{G}_\gl\!\!\upharpoonright_{L^2(\gf_2)}$. Since $\mathcal{G}_\gl\!\!\upharpoonright_{L^2(\gf_2)}$ is positivity improving, we have $|\vgf|>0$. It follows that any such eigenfunction has a definite sign. Consequently, a standard orthogonality argument shows that $(\gm-\gl)^{-1}$ is simple  (cf. \cite[Theorem~XIII.43]{RS}).

 Assume further that  $\gf W\tilde{\gf}\in L^1(\Gw)$, and denote $\hat{\gf}_p:= |\vgf|^{-1}(|\vgf|W\tilde{\gf})^{1/p}$. Then
  by the Cauchy-Schwarz inequality $\vgf W\tilde{\gf}\in L^1(\Gw)$, and hence $|\vgf|$ is a positive eigenfunction
of the operator $\mathcal{G}_\gl\!\!\upharpoonright_{L^2(\hat{\gf}_2)}$. Therefore, by Theorem~\ref{thmbddgen1} {\em (ii)}, $\gm=\gl_0$, the operator $L_{\gl_0}$ is positive-critical, and for some constant $c$ we have $\vgf=c\gf$. In addition,  part {\em (iii)} of Theorem~\ref{thmbdd} implies that $(\gl_0-\gl)^{-1}$ is the unique eigenvalue of $\mathcal{G}_\gl\!\!\upharpoonright_{L^p(\gf_p)}$ with a nonnegative eigenfunction $\gf$ for all $1\leq p < \infty$ and all $\gl<\gl_0$.

The simplicity of $(\gl_0-\gl)^{-1}$ as  an eigenvalue of $\mathcal{G}_\gl\!\!\upharpoonright_{L^p(\gf_p)}$ for $2\leq p \leq \infty$ follows now from the simplicity for $p=2$ and the embedding \eqref{imbedding}.
\end{proof}
Next we study the case $p=1$, and obtain the simplicity of the eigenvalue $(\gl_0-\gl)^{-1}$ for all $p$ under the assumption that $(\gm-\gl)^{-1}$ is an eigenvalue of $\mathcal{G}_\gl\!\!\upharpoonright_{L^1(W\tilde{\gf})}$.  Note that for $p=1$ the result does not depend on a particular positive solution $\gf$ of the equation $L_\gm u=0$ in $\Gw$ since $L^1(\gf_1)=L^1(W\tilde{\gf})$. As a result, we obtain $L^1$- and $L^p$-Liouville theorems for solutions of the equation $L_{\gm}u=0$ which are
eigenfunctions of the operator $\mathcal{G}_\gl$ with an eigenvalue $(\gm-\gl)^{-1}$.
 \begin{theorem}\label{thmbddgenl1}
Let $L$ be an elliptic operator on $\Gw$ of the form \eqref{div_L},
and let $W$ be a positive potential. Fix $\gm\leq\gl_0$. Let $\tilde{\gf}$
be a positive solution of the equation  $L_\gm^\star u=0$ in $\Gw$,
 For $\gl<\gm$ consider the operator
$\mathcal{G}_\gl\!\!\upharpoonright_{L^1(W\tilde{\gf})}$. Suppose further that  $(\gm-\gl)^{-1}$ is an eigenvalue of $\mathcal{G}_\gl\!\!\upharpoonright_{L^1(W\tilde{\gf})}$ with an eigenfunction $\vgf$.

\vspace{2mm}

(i) The eigenfunction $\vgf$ has a definite sign, $\gm=\gl_0$, and $L_{\gl_0}$ is positive-critical with respect to $W$ with a ground state $\vgf$. In particular, $(\gl_0-\gl)^{-1}$ is a simple eigenvalue of
the operator $\mathcal{G}_\gl\!\!\upharpoonright_{L^1(W\tilde{\gf})}$.

\vspace{2mm}

(ii) Set
 $$\vgf_p:= |\vgf|^{-1} (|\vgf|W\tilde{\gf})^{1/p} \qquad 1\leq p\leq \infty.$$
Then for all  $\gl<\gl_0$ and all $1\leq  p\leq \infty$, $\vgf$ is an eigenfunction of $\mathcal{G}_\gl\!\!\upharpoonright_{L^p(\vgf_p)}$ with an eigenvalue $(\gl_0-\gl)^{-1}$. Moreover, $(\gl_0-\gl)^{-1}$ is a simple eigenvalue, and it is the unique eigenvalue with a nonnegative eigenfunction.
\end{theorem}
\begin{proof}
(i) Fix any positive solution $\gf$ of the equation $L_\gm u=0$ in $\Gw$, and let  $\gf_p:= \gf^{-1} (\gf W\tilde{\gf})^{1/p}.$
Clearly, $L^1(\gf_1)= L^1(W\tilde{\gf})$. By Theorem~\ref{thmbddgen0}, for any  $\gl<\gm$ the operator $\mathcal{G}_\gl\!\!\upharpoonright_{L^1(W\tilde{\gf})}$ is bounded with a norm
$\|\mathcal{G}_\gl\!\!\upharpoonright_{L^1(W\tilde{\gf})}\|_{L^1(W\tilde{\gf})}\leq (\gm-\gl)^{-1}$.

Let $\vgf$ be an eigenfunction of $\mathcal{G}_\gl\!\!\upharpoonright_{L^1(W\tilde{\gf})}$ with an eigenvalue $(\gm-\gl)^{-1}$. By our assumption,
\be\label{gseql1}
\frac{|\vgf(x)|}{\gm-\gl}=   \left|\int_{\Omega}\Green{\Omega}{L_\gl}{x}{y}W(y)\vgf(y)\dnu(y)\right|
   \quad\forall x\in \Gw.
 \ee
Therefore, using \eqref{gseqsc} we obtain
\begin{multline}\label{eq_L1}
\frac{1}{\gm-\gl}\int_\Gw |\vgf(x)|\, W(x)\tilde{\gf}(x)\dnu=\\[2mm]
\int_\Gw \left|\int_{\Omega}\Green{\Omega}{L_\gl}{x}{y}W(y)\vgf(y)\dnu(y)\right|W(x)\tilde{\gf}(x)\dnu(x)\leq \\[2mm]
\int_\Gw \left(\int_{\Omega}\Green{\Omega}{L_\gl}{x}{y}W(y)|\vgf(y)|\dnu(y)\right) W(x)\tilde{\gf}(x)\dnu(x) =\\[2mm]
\int_\Gw \left(\int_{\Omega}\Green{\Omega}{L_\gl}{x}{y}W(x)\tilde{\gf}(x)\dnu(x) \right) W(y)|\vgf(y)|\dnu(y) \leq\\[2mm]
\frac{1}{\gm-\gl} \int_{\Omega} \tilde{\gf}(y) W(y)|\vgf(y)|\dnu(y).
\end{multline}
Therefore, the two inequalities in \eqref{eq_L1} are equalities, and in particular,  for almost all $x\in\Gw$ we have
\be\label{samesign}
\left|\int_{\Omega}\Green{\Omega}{L_\gl}{x}{y}W(y)\vgf(y)\dnu(y)\right| =
\int_{\Omega}\Green{\Omega}{L_\gl}{x}{y}W(y)|\vgf(y)|\dnu(y),
\ee
and hence $\vgf$ does not change its sign in $\Gw$. Moreover, the equality in \eqref{eq_L1} implies also that  $\tilde{\gf}$ is an invariant solution of the equation $L_\gm^\star u=0$ in $\Gw$. Since $|\vgf| W \tilde{\gf}\in L^1(\Gw)$,  Remark~\ref{remnoninv} implies that $\gm=\gl_0$,  $L_{\gl_0}$ is positive-critical, and $|\vgf|$ is its ground state. Hence, the simplicity of the `maximal' eigenvalue follows. Consequently, part (ii) follows using the embedding \eqref{imbedding}, part (i) of the present theorem, and part (iii) of Theorem~\ref{thmbdd}.
\end{proof}
\begin{corollary}\label{CorL1}
Let $L$ be an elliptic operator on $\Gw$ of the form \eqref{div_L},
and let $W$ be a positive potential. Assume further that $L^\star \mathbf{1}=0$.
 For $\gl<0$ consider the operator
$\mathcal{G}_\gl\!\!\upharpoonright_{L^1(W)}$.

Suppose that  $|\gl|^{-1}$ is an eigenvalue of $\mathcal{G}_\gl\!\!\upharpoonright_{L^1(W)}$ with an eigenfunction $\vgf$, and set
 $\vgf_p:= |\vgf|^{-1} (|\vgf|W)^{1/p}$, where $1\leq p\leq \infty$.
  Then $\vgf$ has a definite sign, $\gl_0=0$, $L$ is positive-critical, and for all  $\gl<0$ and all $1\leq p\leq \infty$,  $|\gl|^{-1}$ is the unique eigenvalue of $\mathcal{G}_\gl\!\!\upharpoonright_{L^p(\vgf_p)}$ with a nonnegative eigenfunction. Moreover, $|\gl|^{-1}$ is a simple eigenvalue of $\mathcal{G}_\gl\!\!\upharpoonright_{L^p(\vgf_p)}$
\end{corollary}
\begin{example}\label{exL1}
Let $\Gw=\R^d$, and consider a uniformly elliptic operator $L$ with bounded smooth coefficients on $\R^d$ such that $L^\star \mathbf{1}=0$ in $\R^d$ (these conditions can be relaxed). For example, assume that $L$ is of the form  
$$Lu:=  -\mathrm{div} \left(A(x)\nabla u +  u\tilde{b}(x) \right)\qquad x\in\R^d.$$
Suppose that the equation $Lu=0$ in $\R^d$ admits a solution $\vgf$ satisfying $\vgf\in L^1(\R^d)$. Let $k_L^{\R^d}(x,y,t)$ the heat kernel associated with the operator $L$ on $\R^d$. Then
$$v(x,t):=\int_{\R^d} k_L^{\R^d}(x,y,t)\vgf(y)\dy$$
is a well defined $L^1$-solution of the Cauchy problem with the initial condition $\vgf$. Since the uniqueness of the Cauchy problem for $L^1$-initial conditions holds true \cite{AB,ABC}, it follows that  $v=\vgf$.
Fix $\gl<0$.  It follows that
\begin{multline*}\frac{\vgf(x)}{|\gl|}=\int_0^\infty e^{\gl t}v(x,t)\dt=\int_0^\infty e^{\gl t}\left(\int_{\R^d}  k_L^{\R^d}(x,y,t)\vgf(y)\dy\right)\dt=\\
\int_{\R^d}\left(\int_0^\infty  e^{\gl t}k_L^{\R^d}(x,y,t)\dt\right)\vgf(y)\dy=\int_{\R^d}\Green{\R^d}{L_\gl}{x}{y}\vgf(y)\dy.
\end{multline*}
So, $\vgf$ is an eigenfunction of  $\mathcal{G}_\gl\!\!\upharpoonright_{L^1(\R^d)}$ with an eigenvalue $|\gl|^{-1}$. Corollary~\ref{CorL1} implies that  $L$ is positive-critical with respect to $W=\mathbf{1}$, and $|\vgf|>0$ is the corresponding ground state. In particular, $\mathbf{1}$ is an invariant positive solution of the operator $L^\star$. Moreover, $|\gl|^{-1}$ is a simple eigenvalue of $\mathcal{G}_\gl\!\!\upharpoonright_{L^1(\R^d)}$.

Let $1\leq p\leq \infty$,  and $\vgf_p:= |\vgf|^{\frac{1}{p}-1}$.  Then, by Corollary~\ref{CorL1} for all  $\gl<0$ and all $1\leq p\leq \infty$,  $|\gl|^{-1}$ is the unique eigenvalue of $\mathcal{G}_\gl\!\!\upharpoonright_{L^p(\vgf_p)}$ with a nonnegative eigenfunction $|\vgf|$. Moreover, $|\gl|^{-1}$ is a simple eigenvalue of $\mathcal{G}_\gl\!\!\upharpoonright_{L^p(\vgf_p)}$ for all $1\leq p\leq \infty$.
\end{example}
\begin{remark}
For $L^p$-Liouville theorems for symmetric diffusion operators on complete Riemannian manifolds, see \cite{XDLi,Wu} and references therein.
\end{remark}
%%%%%%%%%%%%%%%%%
\section{Semigroups and generators}\label{secsemigr}
\begin{definition}
Let $B$ be a Banach space and $\Gl\subset \mathbb{C}$, and
consider a one-parameter family of operators $\mathcal{J}(\gl)\in \mathcal{L}(B)$ defined for each $\gl\in \Gl$.
The family $\{\mathcal{J}(\gl)\mid \gl\in\Gl\}$ is called a {\em pseudoresolvent}
if
$$\mathcal{J}(\gl)-\mathcal{J}(\gn)=(\gn-\gl)\mathcal{J}(\gl)\mathcal{J}(\gn)$$
holds for all $\gl,\gn\in\Gl$ (see \cite[Definition~4.3]{EN}).
\end{definition}

Let $L$ be an elliptic operator on $\Gw$ of the form \eqref{div_L},
and $W$ a positive potential.  Fix $\gm\leq \gl_0=\gl_0(L,W,\Gw)$, and let
\begin{equation}\label{Gl}
       \Gl:=\left\{
              \begin{array}{ll}
       \{\gl \in \R\mid\gl\leq \gm\}  &\qquad\mbox{if $L_{\gm}$ is subcritical,} \\[4mm]
  \{\gl \in \R\mid\gl<\gm\}  &\qquad\mbox{if $\gm=\gl_0$ and $L_{\gl_0}$ is critical.}
              \end{array}
            \right.
  \end{equation}
Recall that by (2.10) of \cite{Pcrit2}, for all $\gl,\gn\in \Gl$ the
corresponding Green functions
satisfy the (pointwise) resolvent equation
\begin{equation}\label{reseq}
\Green{\Omega}{L_\gl}{x}{y}=
\Green{\Omega}{L_\gn}{x}{y} + (\gl-\gn)\int_{\Omega}
\Green{\Omega}{L_\gl}{x}{z}W(z)\Green{\Omega}{L_\gn}{z}{y}\dnu(z)
\end{equation}
for all $x,y\in \Omega$, $x\neq y$.

Let $\gf$ and $\tilde{\gf}$ be two fixed positive
solutions of the equations $L_{\gm}u=0$ and $L^\star _{\gm}u=0$ in
$\Gw$, respectively. It follows from Theorem~\ref{thmbddgen0} and \eqref{reseq} that for any $1\leq p\leq \infty$, the family
$$\{\mathcal{G}_{(-\gl)}\!\!\upharpoonright_{L^p(\gf_p)}\mid -\gl\in \Gl\}=\{\mathcal{G}_{L+\gl W}\!\!\upharpoonright_{L^p(\gf_p)}\mid -\gl\in \Gl\}$$ is a
pseudoresolvent on $L^p(\gf_p)$.

We claim that for $1\leq p < \infty$ and $\gl\in \Gl$, the range of
$\mathcal{G}_\gl\!\!\upharpoonright_{L^p(\gf_p)}$ is dense in $L^p(\gf_p)$. Indeed, take $u\in C_0^\infty(\Omega)$, and let $N_u\in \N$ be such that $\supp u \Subset \Omega_{N_u}$. Set $f(x):=(W(x))^{-1}L_\gl u(x)$. For $n\geq N_u$ denote $u_n:=\mathcal{G}^{\Omega_n}_\gl\!\!\upharpoonright_{L^p(\gf_p)}f$. Clearly,
$\supp f\subset \Omega_n$ for all $n\geq N_u$. Therefore,  by uniqueness, for any such $n$ we have $u_n=u$ in $\Omega_n$, and consequently, $u=\mathcal{G}^{\Omega}_\gl\!\!\upharpoonright_{L^p(\gf_p)}f$, and $u$ belongs to the range of $\mathcal{G}_\gl\!\!\upharpoonright_{L^p(\gf_p)}$. Since
for $1\leq p < \infty$ the space $C_0^\infty(\Omega)$ is dense in $L^p(\gf_p)$, it follows that the range of
$\mathcal{G}_\gl\!\!\upharpoonright_{L^p(\gf_p)}$ is dense in $L^p(\gf_p)$.

 On the other hand, by Theorem~\ref{thmbddgen},
zero is not an eigenvalue of the operators
$\mathcal{G}_\gl\!\!\upharpoonright_{L^p(\gf_p)}$ for $\gl<\gm$ and all
$1< p\leq \infty$. Moreover, for $1\leq p \leq \infty$ we have
\be\label{eq:1n}\|\mathcal{G}_\gl\|_{L^p(\gf_p)}
\leq\frac{1}{\gm-\gl},
 \ee and in particular,
$$\limsup_{\gl\to -\infty}\|\gl\mathcal{G}_\gl\|_{L^p(\gf_p)} \leq \lim_{\gl\to -\infty}\frac{-\gl}{\gm-\gl}=1.$$
Moreover, \eqref{eq:1n} implies also that
if $\gm\geq 0$ and $\gl<0$, then
$$\|\gl\mathcal{G}_\gl\|_{L^p(\gf_p)}
\leq 1.$$ Therefore, by Proposition~III.4.6 and Corollary~III.4.7
of \cite{EN}, and the Hille-Yosida theorem \cite[Theorem~II.3.5]{EN}, we have:
\begin{theorem} Let $L$ be an elliptic operator on $\Gw$ of the form \eqref{div_L},
and let $W$ be a positive potential. Fix  $1 < p< \infty$,  $\gm\leq
\gl_0$, and let $\Gl$ be as in \eqref{Gl}. Let $\gf$ and $\tilde{\gf}$ be two fixed positive
solutions of the equations $L_{\gm}u=0$ and $L^\star _{\gm}u=0$ in
$\Gw$, respectively.

(i) The pseudoresolvent family
$$\{\mathcal{G}_{(-\gl)}\!\!\upharpoonright_{L^p(\gf_p)}\mid
-\gl\in\Gl \}$$ is a resolvent of a densely defined closed operator
$$A_p:=-\left(\mathcal{G}_{\gl_1}\!\!\upharpoonright_{L^p(\gf_p)}\right)^{-1}-\gl_1 \qquad \mbox{where }\gl_1\in \Gl$$ on $L^p(\gf_p)$ with a domain $D(A_p)=R(\mathcal{G}_{\gl_1}\!\!\upharpoonright_{L^p(\gf_p)})$. In particular,
$A_p=-\frac{1}{W}L$ on $\core$. Moreover,
$(-\gm,\infty)\subset \gr(A_p)$, and for $\gl\in (-\gm,\infty)$ we have
$$R(\gl, A_p):=(\gl-A_p)^{-1}=\mathcal{G}_{(-\gl)}\!\!\upharpoonright_{L^p(\gf_p)}.$$

(ii) Zero is not an eigenvalue of $\mathcal{G}_\gl\!\!\upharpoonright_{L^p(\gf_p)}$.

(iii) If $\gm\geq 0$, then $(A_p, D(A_p))$ generates a strongly
continuous contraction semigroup. Moreover, for every $\gl\in
\mathbb{C}$ with $\mathrm{Re}\,\gl>0$ one has $\gl\in \gr(A_p)$,
and
$$\|R(\gl,A_p)\|_{L^p(\gf_p)}\leq \frac{1}{\mathrm{Re}\,\gl}\,.$$
\end{theorem}
%%%%%%%%%%%%%%%%%%%%%
\section{Compactness and semismall perturbations}\label{secanti}
%%%%%%%%%%%%%%%%%%%%%%%%%%%%%%%%%%%%%%%%%%%%%%%%%%%
 Throughout this section, we assume that $L$ is subcritical in $\Gw$ and $W>0$
is a {\em semismall perturbation} of $L$ and $L^\star $ in
$\Gw$. By Remark~\ref{remspert}, $\gl_0 > 0$, and  $L_{\gl_0}$ is positive-critical.  Denote by $\gf$ and $\tilde{\gf}$ the ground states of $L_{\gl_0}$
and $L^\star_{\gl_0} $, respectively. We may assume that $\gf(x_0) = 1$.  So,
 \be \label{phipl}
(L-\gl_0W)\gf = (L^\star -\gl_0W)\tilde{\gf} = 0\; \mbox{ in } \Gw \quad \mbox{and } \tilde{\gf}W\gf\in L^1(\Gw).
\end{equation}
Without loss of generality, we may assume that \be\label{phiwphi} \int_\Gw
\tilde{\gf}(x)W(x)\gf(x)\dnu(x)=1.
\end{equation}
It follows that
$\tilde{\gf} \in (L^1(\tilde{\gf}_1))\subset
(L^\infty(\gf_\infty))^\star $, and for every $f\in
L^\infty(\gf_\infty)$ we have $\langle \tilde{\gf},f
\rangle=\int_\Gw \tilde{\gf}(x)W(x)f(x)\dnu$.

\medskip

The aim of the present section is to prove that under the above assumptions, the
integral operator $\mathcal{G}_\gl$ is compact on $L^p(\gf_p)$ for any $1\leq
p\leq \infty$ and $\gl<\gl_0$, and its spectrum is $p$-independent. We first prove the compactness of $\mathcal{G}_\gl$.

\begin{theorem}\label{thmcomp}
 Let $L$ be a subcritical operator in $\Gw$. Assume that $W>0$ is a semismall perturbation of $L^\star $ and $L$ in $\Gw$. Then for any $1\leq
p\leq \infty$ and $\gl<\gl_0$, the integral operators
\begin{align*}
\mathcal{G}_\gl f(x)=  \int_{\Omega}  \Green{\Omega}{L-\gl
W}{x}{y}W(y)f(y)\dnu(y),\\[2mm] 
\mathcal{G}_\gl^\odot f(y) =
 \int_{\Omega}  \Green{\Omega}{L-\gl W}{x}{y}W(x)f(x)\dnu(x)
 \end{align*}
are compact on $L^{p}(\gf_p)$ and $L^{p}(\tilde{\gf}_p)$,
respectively.
\end{theorem}
\begin{proof} By \cite[Theorem~5.1]{P99}, the operators
$\mathcal{G}_\gl\!\!\upharpoonright_{L^{\infty}(\gf_\infty)}$ and
$\mathcal{G}_\gl^\odot\!\!\upharpoonright_{L^{\infty}(\tilde{\gf}_\infty)}$ are
compact on $L^{\infty}(\gf_\infty)$ and
$L^{\infty}(\tilde{\gf}_\infty)$, respectively. For the sake of completeness, we prove the compactness of $\mathcal{G}_\gl\!\!\upharpoonright_{L^{\infty}(\gf_\infty)}$; the proof of the compactness of $\mathcal{G}_\gl^\odot\!\!\upharpoonright_{L^{\infty}(\tilde{\gf}_\infty)}$ is identical.

Let $\{f_n\}$ be a bounded sequence in $L^{\infty}(\gf_\infty)$. By Theorem~\ref{thmbddgen0}, the sequence
$u_n:=\mathcal{G}_\gl\!\!\upharpoonright_{L^{\infty}(\gf_\infty)}f_n$ is bounded in $L^{\infty}(\gf_\infty)$, and satisfies
$$|u_n(x)|\leq \int_{\Omega}\Green{\Omega} {L_{\gl}}{x}{y}W(y)|f_n(y)|\dnu(y)\leq
C\gf_0(x),$$ where $C:= (\gl_0-\gl)^{-1}\sup_n \|f_n\|_{\infty,\gf_\infty}$ is independent of $n$.
Moreover, it follows that $u_n$ is the unique function in $L^{\infty}(\gf_\infty)$ which is a (weak) solution of the
equation $L_\gl u=f_n$ in $\Gw$ (cf. \cite[Theorem~4.6]{P99}). Consequently, a standard elliptic argument implies that
the sequence $\{u_n\}$ admits a subsequence which
converges in the compact open topology to a function $u$. Clearly, $\|u\|_{\infty,\gf_\infty}\leq C$, so, $u\in L^{\infty}(\gf_\infty)$

Since $W$ is a semismall perturbation, it follows that for any given $\vge>0$ there exists $K$ such that for any $k\geq K$
and $n,m\in \Nat$
\begin{multline}\label{extest}
 \int_{\Omega_k^*}\Green{\Omega}{L_\gl}{x}{y}W(y)|f_n(y)-f_m(y)|\dnu(y) \leq\\[2mm]
2C\int_{\Omega_k^*}\Green{\Omega}{L_\gl}{x}{y}W(y)\gf_0(y)\dnu(y)<\vge\gf_0(x) \qquad \forall x\in \overline{\Gw_k^*},
\end{multline}
and by the generalized maximum principle in $\Gw_k$, \eqref{extest} holds for any $x\in\Gw$.

The local
uniform convergence of $\{u_n\}$ implies that there exists $N_\vge\in \Nat$ such that
$|u_n-u_m|\leq \vge\gf_0$ in $\Gw_K$ for all $n,m \geq N_\vge$.

Fix $n,m \geq N_\vge$. It follows from \cite[Lemma~4.3]{P99} and the linearity that on $\Gw_K^*$ we have
$$u_n(x)-u_m(x)=h_{n,m}(x)+ \int_{\Omega_K^*}\Green{\Omega^*_K}{L_\gl}{x}{y}W(y)(f_n(y)-f_m(y))\dnu(y),$$
where $h_{n,m}\in L^{\infty}(\gf_\infty)$ satisfies
$$L_\gl h_{n,m}=0 \quad \mbox{in } \Gw_K^*, \quad \mbox{and} \quad h_{n,m}(x)=u_n(x)-u_m(x) \quad x\in \partial \Gw_K.$$
Since $|h_{n,m}|\leq \vge\gf_0$ on $\partial \Gw_K$, and $h_{n,m}$ has minimal growth in $\Gw$, it follows that $|h_{n,m}|\leq 2\vge\gf_0$ in $\Gw_K^*$. On the other hand, by \eqref{extest} we have
\begin{multline*}
\left|\int_{\Omega_K^*}  \Green{\Omega^*_K}{L_\gl}{x}{y}W(y)(f_n(y)-f_m(y)) \dnu(y) \right|  \leq\\[2mm]
\int_{\Omega_K^*}    \Green{\Omega}{L_\gl}{x}{y}W(y)|f_n(y)-f_m(y)| \dnu(y)  <
\vge\gf_0(x) \qquad \forall x\in \Omega_K^*.
\end{multline*}

Consequently,  we infer that
$|u_n-u_m|\leq 3\,\vge\gf_0$ in $\Gw^*_K$ for all $n,m\geq N_\vge$. Thus,
$u_n \to u$ in $L^{\infty}(\gf_\infty)$.

\medskip

Since for each
$1\leq p\leq \infty$ the operator
$\mathcal{G}_\gl\!\!\upharpoonright_{L^{p}(\gf_{p})}$  is bounded on $L^{p}(\gf_{p})$, and $\mathcal{G}_\gl\!\!\upharpoonright_{L^{\infty}(\gf_\infty)}$ is
compact on $L^{\infty}(\gf_\infty)$,  it
follows from a variant of the Riesz-Thorin interpolation theorem
with respect to compact operators \cite[Theorem~1.1]{Cwikel} that
$\mathcal{G}_\gl\!\!\upharpoonright_{L^{p}(\gf_p)}$ and are compact for
all $1\leq p\leq \infty$. The same is true for $\mathcal{G}_\gl^\odot\!\!\upharpoonright_{L^p(\tilde{\gf}_p)}$.
\end{proof}
%%%%%%%%%%%%%%%%%%%
\begin{remark}\label{remSchauder}
For $\gl<\gl_0$ the operators
$\mathcal{G}_\gl\!\!\upharpoonright_{L^{\infty}(\gf_\infty)}$ and
$\mathcal{G}_\gl^\odot\!\!\upharpoonright_{L^{\infty}(\tilde{\gf}_\infty)}$ are the
dual operators of
$\mathcal{G}_\gl^\odot\!\!\upharpoonright_{L^{1}(\tilde{\gf}_{1})}$ and
$\mathcal{G}_\gl\!\!\upharpoonright_{L^{1}(\gf_{1})}$, respectively. Therefore, the
well-known Schauder theorem directly implies that
$\mathcal{G}_\gl^\odot\!\!\upharpoonright_{L^{1}(\tilde{\gf}_{1})}$ and
$\mathcal{G}_\gl\!\!\upharpoonright_{L^{1}(\gf_{1})}$ are compact on
$L^{1}(\tilde{\gf}_1)$ and $L^{1}(\gf_1)$, respectively.
\end{remark}
%%%%%%%%%%%%%%%%%%
\begin{remark}\label{Albert}
In the proof of Theorem~\ref{thmcomp} we used the fact that {\em real} interpolation preserves the compactness of an operator.
We recall that in his remarkable paper \cite{AC} A.~Calder\'on implicitly asked a question which is apparently still open today: Does {\em complex} interpolation preserve the compactness of an operator? For a recent survey on this question see \cite{MC}.
\end{remark}
%%%%%%%%%%%%%%%%%%%%%%%%%%%%%%%%
The next theorem discusses the spectral properties of $\mathcal{G}_\gl\!\!\upharpoonright_{L^{p}(\gf_p)}$.
\begin{theorem}\label{corpindepnd}
Under the assumptions of Theorem~\ref{thmcomp} we have:
\begin{enumerate}
\item For $1\leq p\leq \infty$, the spectrum of $\mathcal{G}_\gl\!\!\upharpoonright_{L^{p}(\gf_p)}$ contains $0$, and besides, consists of at most a sequence of eigenvalues of
finite multiplicity which has no point of accumulation except $0$.

\item For any $1\leq p\leq\infty$, $\gf$ (resp.    $\tilde{\gf}$) is the unique nonnegative  eigenfunction of the operator
$\mathcal{G}_\gl\!\!\upharpoonright_{L^p(\gf_p)}$ (resp.,
$\mathcal{G}_\gl^\odot\!\!\upharpoonright_{L^p(\tilde{\gf}_p)}$). The corresponding eigenvalue
$\gn=(\gl_0-\gl)^{-1}$ is simple.

\item The spectrum of $\mathcal{G}_\gl\!\!\upharpoonright_{L^{p}}(\gf_p)$  is
$p$-independent for all $1\leq p\leq \infty$, and we have
$$0\in \gs\left(\mathcal{G}_\gl\!\!\upharpoonright_{L^p(\gf_p)}\right) = \gs\left(\mathcal{G}_\gl^\odot\!\!\upharpoonright_{L^p(\tilde{\gf}_p)}\right)
 \subset\overline{ B\Big(0, (\gl_0-\gl)^{-1}\Big)} .$$

\end{enumerate}
\end{theorem}
\begin{proof} (1) The characterization of the spectrum of
$\mathcal{G}_\gl\!\!\upharpoonright_{L^{p}(\gf_p)}$ for each $p$ follows from the
Riesz-Schauder theory for compact operators.

\medskip

(2) Follows from Theorem~\ref{thmbddgenl1}.

\medskip

(3) The compactness of all the operators $\mathcal{G}_\gl\!\!\upharpoonright_{L^{p}(\gf_p)}$ implies that it is enough to show that $\gs_{\mathrm{point}}(\mathcal{G}_\gl\!\!\upharpoonright_{L^{p}(\gf_p)})$ is $p$-independent.

By \eqref{imbedding}
we have for any $1\leq p\leq \infty$ that  \be\label{inclusion}
\gs_{\mathrm{point}}(\mathcal{G}_\gl\!\!\upharpoonright_{L^{\infty}(\gf_\infty)})\subset
\gs_{\mathrm{point}}(\mathcal{G}_\gl\!\!\upharpoonright_{L^{p}(\gf_p)}) \subset
\gs_{\mathrm{point}}(\mathcal{G}_\gl\!\!\upharpoonright_{L^{1}(\gf_1)}), \ee and
\be\label{inclusion1}
\gs_{\mathrm{point}}(\mathcal{G}^\odot_\gl\!\!\upharpoonright_{L^{\infty}(\tilde{\gf}_\infty)})\subset
\gs_{\mathrm{point}}(\mathcal{G}^\odot_\gl\!\!\upharpoonright_{L^{p}(\tilde{\gf}_p)})
\subset
\gs_{\mathrm{point}}(\mathcal{G}^\odot_\gl\!\!\upharpoonright_{L^{1}(\tilde{\gf}_1)}).
\ee
On the other hand,  $\gf$ and $\tilde{\gf}$ are $\gl_0$-invariant positive solutions of the operator $L$ and $L^\star$, respectively. Therefore, Theorem~\ref{thmbddgen1} implies that $\|\mathcal{G}_\gl\|_{L^p(\gf_p)}= (\gm-\gl)^{-1}$.

Recall that by Theorem~\ref{thmbddgen0}, for  any $1\leq p< \infty$, the operator
$\mathcal{G}_\gl^\odot\!\!\upharpoonright_{L^{p'}(\tilde{\gf}_{p'})}$ is the dual
operator of $\mathcal{G}_\gl\!\!\upharpoonright_{L^{p}(\gf_{p})}$, and
$\mathcal{G}_\gl\!\!\upharpoonright_{L^{p'}(\gf_{p'})}$ is the dual
 of $\mathcal{G}_\gl^\odot\!\!\upharpoonright_{L^{p}(\tilde{\gf}_{p})}$.
Since the spectra of a bounded operator and its dual are
equal,  we have
$$\gs(\mathcal{G}_\gl\!\!\upharpoonright_{L^{1}(\gf_{1})})=
\gs(\mathcal{G}_\gl^\odot\!\!\upharpoonright_{L^{\infty}(\tilde{\gf}_{\infty})}),
\qquad \gs(\mathcal{G}_\gl^\odot\!\!\upharpoonright_{L^{1}(\tilde{\gf}_{1})})=
 \gs(\mathcal{G}_\gl\!\!\upharpoonright_{L^{\infty}(\gf_{\infty})}).$$
Thus all the point-spectra in \eqref{inclusion} and \eqref{inclusion1}
are equal.
\end{proof}
%%%%%%%%%%%%%%%%%%%%%%%%%%%%%%%%%%%%%%%%%%%%

\end{document}